\documentclass{amsart}
\usepackage{amsfonts}
\usepackage{amsmath}

\newtheorem{theorem}{Theorem}
\theoremstyle{plain}

\newtheorem{corollary}{Corollary}

\newtheorem{lemma}{Lemma}

\newtheorem{remark}{Remark}

\numberwithin{equation}{section}
\input{tcilatex}
\input tcilatex

\begin{document}
\def\cprime{$'$}
\def\cprime{$'$}

\title{Parabolic comparison revisited and applications}
\author{Joscha Diehl, Peter K. Friz and Harald Oberhauser}
\address{JD and HO are affiliated to TU Berlin. PKF is corresponding author
(friz@math.tu-berlin.de) and affiliated to TU and WIAS Berlin.}

\begin{abstract}
We consider the Cauchy-Dirichlet problem%
\begin{equation*}
\partial _{t}u-F\left( t,x,u,Du,D^{2}u\right) =0\text{ on }(0,T)\times 
\mathbb{R}^{n}
\end{equation*}%
in viscosity sense. Comparison is established for bounded semi-continuous
(sub-/super-)solutions under structural assumption (3.14) of the User's
Guide plus a mild condition on $F$ such as to cope with the unbounded
domain. Comparison on $(0,T]$, space-time regularity and existence are also
discussed. Our analysis passes through an extension of the parabolic theorem
of sums which appears to be useful in its own right.
\end{abstract}

\keywords{parabolic viscosity PDEs, theorem of sums, regularity of viscosity
solutions}
\maketitle

\section{Introduction}

We recall some basic ideas of (second order) viscosity theory (Crandall,
Ishii,\ Lions ... \cite{MR1118699UserGuide, MR2179357FS}). Consider a
real-valued function $u=u\left( x\right) $ with $x\in \mathbb{R}^{n}$ and
assume $u\in C^{2}$ is a classical supersolution,%
\begin{equation*}
-G\left( x,u,Du,D^{2}u\right) \geq 0,
\end{equation*}%
where $G$ is a (continuous) function, \textit{degenerate elliptic} in the
sense that $G\left( x,u,p,A\right) \leq G\left( x,u,p,A+B\right) $ whenever $%
B\geq 0$ in the sense of symmetric matrices, one also requires that $G$ is
non-increasing in $u$; under these assumptions $G$ is called \textit{proper}%
. The idea is to consider a (smooth) test function $\varphi $ which touches $%
u$ from below at some point $\bar{x}$. Basic calculus implies that $Du\left( 
\bar{x}\right) =D\varphi \left( \bar{x}\right) ,\,D^{2}u\left( \bar{x}%
\right) \geq D^{2}\varphi \left( \bar{x}\right) $ and, from degenerate
ellipticity,%
\begin{equation}
-G\left( \bar{x},\varphi ,D\varphi ,D^{2}\varphi \right) \,\geq 0.
\label{Gxbar}
\end{equation}%
This suggests to define a \textit{viscosity subsolution} (at the point $\bar{%
x}$) to $-G=0$ as a (upper semi-)continuous function $u$ with the property
that (\ref{Gxbar}) holds for any test function which touches $u$ from above
at $\bar{x}$. Similarly, \textit{viscosity supersolutions} are (lower
semi-)continuous functions, defined via testfunctions touching $u$ from
below and by reversing the inequality in (\ref{Gxbar}); \textit{viscosity
solutions} are both super- and subsolutions (and hence continuous).

Observe that this definition covers (completely degenerate) first order
equations as well as parabolic equations, e.g. by considering $\partial
_{t}-F=0$ where $F$ is proper. The resulting theory (existence, uniqueness,
stability, ...) is without doubt one of most important recent developments
in the field of partial differential equations. In particular, much is known
about the Cauchy-Dirichlet problem%
\begin{equation}
\partial _{t}u-F\left( t,x,u,Du,D^{2}u\right) =0\text{ on }\left( 0,T\right)
\times \Omega  \label{partialu_minus_Fu_Intro}
\end{equation}%
with (nice) initial data, say $u_{0}\in C\left( \Omega \right) $, on some 
\textit{bounded} domain\ $\Omega $; see e.g. Theorem 8.2 in the User's Guide 
\cite{MR1118699UserGuide}. Under structural assumptions on $F$ there is
existence and uniqueness (in some class). In fact, uniqueness follows from a
stronger property known as \textit{comparison:} assume $u$ (resp. $v$) is
are semicontinuous sub- (resp. super) solution and $u_{0}\leq v_{0}$; then $%
u\leq v$ on $\left( 0,T\right) \times \Omega $.

Surprisingly perhaps, much less has been written about the Cauchy-Dirichlet
problem on unbounded domains. This seems to be particularly unfortunate
since much of the recent applications from stochastics are naturally on
unbounded domains\footnote{%
Leaving aside standard examples from stochastic control, let us mention
2BSDEs \cite{CSTV} and stochastic viscosity theory \cite{MR1959710,
MR1647162, MR1799099, MR1807189}; a related rough path point \cite{lyons-98,
lyons-qian-02, friz-victoir-book} was introduced in \cite{CFO} and also
relies on viscosity methods.}. Let us be specific.

(i) We are unaware of a precise result that gives the simplest set of
addtional structural assumptions on $F$ such as to generalize the
aforementioned Theorem 8.2. to, say, bounded solutions on $\left( 0,T\right)
\times \mathbb{R}^{n}$.\newline
(ii) Comparison should be valid up to time $T$; after all $T\times \mathbb{R}%
^{n}$ is not part of the parabolic boundary.\newline
(iii) When does bounded uniformly continuous initial data, $u_{0}\in \mathrm{%
BUC}\left( \mathbb{R}^{n}\right) $, lead to a modulus of continuity of $%
u\left( t,\cdot \right) $, uniformly in $t\in \left[ 0,T\right] $\thinspace ?%
\newline
(iv) When do we have a space-time modulus or, say, a solution $u\in \mathrm{%
BUC}\left( [0,T]\times \mathbb{R}^{n}\right) $\thinspace ?

There are partial answers to these things in the literature of course. Let
us mention in particular \cite{MR1119185} (towards (i) and (iii)) and \cite%
{CGG91} (and the references therein\footnote{%
The authors also point out various mistakes in previous papers in this
context.}) concerning (ii). In the first order case, much can be found in
the books \cite{MR1613876, MR1484411}.

The \textbf{contribution of this paper} is to provide such results (with
fully detailed proofs) in the generality of (\ref{partialu_minus_Fu_Intro}).
While some "general ideas" are without doubt part of the folklore of the
subject (e.g. "spatial modulus follows from comparison", "time modulus
follows from spatial modulus") their proper implementation is far from
trivial. In particular, we were led to an extension of the parabolic theorem
of sums which seems to be quite useful in its own right. To elaborate on
this point, recall that almost every modern treatise of second order
comparison relies in one way or another on the \textit{theorem of sums}
(TOS), also known as \textit{Crandall-Ishii lemma} \cite{MR1073054}. A
parabolic version of the TOS on $\left( 0,T\right) \times \Omega $ then
underlies most second order (parabolic) comparison results; such as those in 
\cite[Chapter 8]{MR1118699UserGuide} or \cite[Chapter 5]{MR2179357FS}. As is
well-known, its application requires a barrier at time $T$; e.g. replace a
subsolution $u$ by $u^{\gamma }:=$ $u-\gamma /\left( T-t\right) $ or so,
followed by $\gamma \downarrow 0$ in the end. In many application this
simple tricks works perfectly fine; sometimes, however, it makes life
difficult. For instance, if $u$ is assumed to be bounded, the same is not
true for $u^{\gamma }$ (altough it is bounded from above); consequently one
may have to introduce various localizations of the non-linearity to deal
with the resulting unboundedness. (An example of the resulting complication
is seen in \cite{DF}.) Concerning the present paper, establishing a spatial
modulus of solutions with the (standard) form of the parabolic theorem of
sums would have led to a (apriori)\ dependence of the spatial modulus in
time; establishing the (desired) uniformity in $t\in \left[ 0,T\right] $,
cf. (iii) above, then entails a painstaking checking of uniformity in $%
\gamma $ for all double limits in the technical lemma \ref%
{LemmaPenalityOffDiag2} below. All these difficulties can be avoided by our
extension of the (parabolic) TOS which remains valid for $t=T$. Perhaps,
from a "general point of view", this is not surprising (after all, the
elliptic TOS\ holds in great generality for locally compact domains and the
parabolic TOS, in a sense, just discards unwanted second order information
related to the $t$ variable) but then, here again, a proper implementation
with full details is quite involved.

\bigskip \textbf{Acknowledgement:}

J. Diehl is supported by an IRTG (Berlin-Zurich) PhD-scholarship; P. Friz
and H. Oberhauser received funding from the European Research Council under
the European Union's Seventh Framework Programme (FP7/2007-2013) / ERC grant
agreement nr. 258237. P. Friz would like to thank G. Barles for a very
helpful email exchange and M. Soner for kindly looking over a earlier
version of this note.

\section{Structural conditions on $F$\label{StructAssF}}

Let $F=F\left( t,x,u,p,X\right) :\left[ 0,T\right] \times \mathbb{R}%
^{n}\times \mathbb{R\times R}^{n}\times S^{n}\rightarrow \mathbb{R}$ be
continuous and degenerate elliptic i.e. non-decreasing in $X$. Assume also
that there exists $\gamma $ such that, uniformly in $t,x,p,X$,%
\begin{equation*}
\gamma \left( u-v\right) \leq F\left( t,x,v,p,X\right) -F\left(
t,x,u,p,X\right) \text{ whenever }v\leq u\text{.}
\end{equation*}%
When $\gamma \geq 0$ such $F$s are called \textit{proper}. Since we will be
interested in parabolic problems of the form $\partial _{t}-F$ a suitable
change of variable ($u\leftrightarrow e^{\gamma t}u$) shows that $\gamma <0$
does not cause trouble. Assume furthermore that there exists, for all $R>0$,
a function $\theta _{R}:\left[ 0,\infty \right] \rightarrow \left[ 0,\infty %
\right] $ with $\theta_R \left( 0+\right) =0$, such that 
\begin{equation}
F\left( t,x,r,\alpha \left( x-\tilde{x}\right) ,X\right) -F\left( t,\tilde{x}%
,r,\alpha \left( x-\tilde{x}\right) ,Y\right) \leq \theta _{R}\left( \alpha
\left\vert x-\tilde{x}\right\vert ^{2}+\left\vert x-\tilde{x}\right\vert
\right)  \label{UG314_1}
\end{equation}%
for all $t\in \left[ 0,T\right] ,\,x,\tilde{x}\in \mathbb{R}^{n},\,r\in %
\left[ -R,R\right] ,$\thinspace $\alpha >0$ and $X,Y\in S^{n}$ (the space of 
$n\times n$ symmetric matrices) which satisfy%
\begin{equation}
-3\alpha 
\begin{pmatrix}
I & 0 \\ 
0 & I%
\end{pmatrix}%
\leq 
\begin{pmatrix}
X & 0 \\ 
0 & -Y%
\end{pmatrix}%
\leq 3\alpha 
\begin{pmatrix}
I & -I \\ 
-I & I%
\end{pmatrix}%
\text{.}  \label{MatrixInequality}
\end{equation}

Under these conditions, comparison for the Cauchy-Dirichlet problem $%
\partial _{t}-F=0$ on $\left( 0,T\right) \times \Omega $, with $\Omega $
bounded, holds (User's Guide, chapter 8). We shall be interested in
comparison for bounded (semi-continuous, sub- and super-) solutions on$%
\,(0,T]\times $ $\mathbb{R}^{n}$. In particular, the unboundedness of $%
\mathbb{R}^{n}$ will require the following additional assumption: assume $%
F=F\left( t,x,u,p,X\right) $ is \textit{uniformly continuous} (\textrm{UC})
whenever $u,p,X$ remain bounded; i.e.%
\begin{equation}
\forall R>0:F|_{\left[ 0,T\right] \times \mathbb{R}^{n}\times \left[ -R,R%
\right] \times B_{R}\times M_{R}}\text{ is uniformly continuous}
\label{UCassumption}
\end{equation}%
where $B_{R},M_{R}$ denote (open)\ balls of radius $R$ in $\mathbb{R}^{n},$%
\thinspace $S^{n}$ respectively.

\section{Parabolic comparison - statement of theorem}

\begin{theorem}
\label{ThmMain}Assume $F$ satisfies the assumptions of section \ref%
{StructAssF}. Consider $u\in \mathrm{bUSC}\left( [0,T)\times \mathbb{R}%
^{n}\right) ,$ $v\in \mathrm{bLSC}\left( [0,T)\times \mathbb{R}^{n}\right) $%
, extended to $\left[ 0,T\right] \times \mathbb{R}^{n}$ via their
semi-continuous envelopes; that is,%
\begin{equation*}
u\left( T,x\right) =\limsup_{\substack{ \left( t,y\right) \in \lbrack
0,T)\times \mathbb{R}^{n}:  \\ t\rightarrow T,y\rightarrow x}}u\left(
t,y\right) ,\,\,\,v\left( T,x\right) =\liminf_{\substack{ \left( t,y\right)
\in \lbrack 0,T)\times \mathbb{R}^{n}:  \\ t\rightarrow T,y\rightarrow x}}%
v\left( t,y\right) .
\end{equation*}%
Assume that, in the sense of parabolic viscosity sub- and super-solutions%
\footnote{%
As is well-known, the precise meaning of (\ref{SubSuperAssumption}) is
expressed (equivalently) in terms of "touching" test-functions or in term of
sub- and super-jets. We shall switch between these points without further
comments.}%
\begin{equation}
\partial _{t}u-F\left( t,x,u,Du,D^{2}u\right) \leq 0\leq \partial
_{t}v-F\left( t,x,v,Dv,D^{2}v\right) \text{ \ on }(0,T)\times \mathbb{R}^{n}.
\label{SubSuperAssumption}
\end{equation}%
Then the following statements hold true.\newline
(i)\ The validity of (\ref{SubSuperAssumption}) extends to $Q:=(0,T]\times 
\mathbb{R}^{n}$ (which reflects that $\left\{ T\right\} \times \mathbb{R}%
^{n} $ is not part of the parabolic boundary of $Q$).\newline
(ii) If $F$ satisfies the structural condition of the previous section, $%
u_{0}:=u\left( 0,\right) ,\,\,v_{0}:=v\left( 0,\right) \in \mathrm{BUC}%
\left( \mathbb{R}^{n}\right) $ and 
\begin{equation*}
u_{0}\leq v_{0}\text{ \ on }\mathbb{R}^{n}
\end{equation*}%
one has the key estimate, valid for all $\left( t,x,y\right) \in \lbrack
0,T]\times \mathbb{R}^{n}\times \mathbb{R}^{n}$,%
\begin{equation*}
u\left( t,x\right) -v\left( t,y\right) \leq \inf_{\alpha }\left[ \frac{%
\alpha }{2}\left\vert x-y\right\vert ^{2}+l\left( \alpha \right) \right]
\end{equation*}%
where $l\left( \alpha \right) $ tends to $0$ as $\alpha \uparrow \infty $,
uniformly in $t\in \left[ 0,T\right] $.
\end{theorem}

\begin{remark}
Since $\tilde{u}\left( t,x\right) =e^{-\gamma t}u\left( t,x\right) $ [resp. $%
\,\tilde{v}\left( t,x\right) =e^{-\gamma t}v\left( t,x\right) $ ] is a sub-
[resp. super-]solution to $\left( \partial _{t}-\tilde{F}\right) \tilde{u}%
+\gamma \tilde{u}=0$ with 
\begin{equation*}
\tilde{F}\left( t,x,p,X\right) =e^{-\gamma t}F\left( t,x,e^{\gamma t}\tilde{u%
},e^{\gamma t}D\tilde{u},e^{\gamma t}D^{2}\tilde{u}\right)
\end{equation*}%
we can always reduce to the case that $\gamma >0$. In particular, we shall
give the proof under this assumption.
\end{remark}

\begin{remark}
The key estimate implies immediately comparison (take $x=y$)%
\begin{equation*}
u\leq v\text{ on }[0,T]\times \mathbb{R}^{n}\text{.}
\end{equation*}%
By a $2\epsilon $ argument, it also yields a spatial modulus for any
solution $u$; uniform in $t\in \left[ 0,T\right] $. Indeed, for fixed $t\leq
T$ pick $\alpha $ large enough so that $l\left( \alpha \right) <\epsilon /2$%
; for any $x,y:\left\vert x-y\right\vert $ small enough (only depending on $%
\alpha $ and hence $\epsilon $) we have $u\left( t,x\right) -u\left(
t,y\right) <\epsilon $. By switching the roles of $x$ and $y$, if necessary,
we see $\left\vert u\left( t,x\right) -u\left( t,y\right) \right\vert
<\epsilon $.
\end{remark}

\section{Parabolic Comparison: Proof of (i)}

Assume $u\in \mathrm{bUSC}\left( [0,T)\times \mathbb{R}^{n}\right) $ solves $%
\partial _{t}u-F\left( t,x,u,Du,Du\right) \leq 0$ with "properness" $\gamma
\geq 0$; with initial data $u\left( 0,\cdot \right) $ on $\left( 0,T\right)
\times \mathbb{R}^{n}$. Extend $u$ to $\mathrm{bUSC}\left( [0,T]\times 
\mathbb{R}^{n}\right) $ by setting%
\begin{equation*}
u\left( T,x\right) =\lim \sup_{t\uparrow T,y\rightarrow x}u\left( t,y\right)
\end{equation*}%
Assume $u-\phi $ has a (strict) max at $\left( T,\bar{x}\right) $, relative
to $\left[ 0,T\right] \times \mathbb{R}^{n}$. (The test function $\phi $ is
defined in an open neighbourhood of $\left[ 0,T\right] \times \mathbb{R}^{n}$%
.) Claim that%
\begin{equation*}
\partial _{t}\phi \left( T,\bar{x}\right) -F\left( T,\bar{x},u\left( T,\bar{x%
}\right) ,D\phi \left( T,\bar{x}\right) ,D^{2}\phi \left( T,\bar{x}\right)
\right) \leq 0.
\end{equation*}

\textbf{Proof: } Take $\left( t^{n},x^{n}\right) \in \left( 0,T\right)
\times \mathbb{R}^{n}$ s.t. $\left( t^{n},x^{n}\right) \rightarrow \left( T,%
\bar{x}\right) $ and $u\left( t^{n},x^{n}\right) \rightarrow u\left( T,\bar{x%
}\right) $. Set $\alpha _{n}:=T-t^{n}\downarrow 0$. Then take%
\begin{equation*}
\left( t_{n},x_{n}\right) \in \arg \max \left( u-\phi -\frac{\alpha _{n}^{2}%
}{T-t}\right) \equiv \arg \max \psi _{n}\text{.}
\end{equation*}%
over $[0,T]\times \mathbb{R}^{n}$. In order to guarantee that the sequence $%
\left( t_{n},x_{n}\right) \in \lbrack 0,T)\times \mathbb{R}^{n}$ remains in
a compact, say $[T/2,T]\times \bar{B}_{1}(\bar{x})$, we make the assumption
(without loss of generality) that $\phi (T,\bar{x})=0$ and $\phi
(t,x)>3|u|_{\infty }$ for $(t,x)\notin \lbrack T/2,T]\times \bar{B}_{1}(\bar{%
x})$; this implies $\left( t_{n},x_{n}\right) \in \lbrack T/2,T]\times \bar{B%
}_{1}(\bar{x})$ for $n$ large enough, as desired. By compactness, $\left(
t_{n},x_{n}\right) \rightarrow \left( \tilde{t},\tilde{x}\right) $ at least
along a subsequence $n\left( k\right) $. We shall run through the other
sequence $\left( t^{n},x^{n}\right) $ along the same subsequence and relabel
both to keep the same notation. Note $\psi _{n}\left( t_{n},x_{n}\right) $
is non-decreasing and bounded, hence%
\begin{equation*}
\psi _{n}\left( t_{n},x_{n}\right) \rightarrow l.
\end{equation*}%
Since $\psi _{n}\left( t_{n},x_{n}\right) \leq \left( u-\phi \right) \left(
t_{n},x_{n}\right) $ it follows (using USC of $u-\phi $) that%
\begin{equation*}
l\leq \left( u-\phi \right) \left( \tilde{t},\tilde{x}\right)
\end{equation*}%
On the other hand,%
\begin{equation*}
\psi _{n}\left( t_{n},x_{n}\right) \geq \psi _{n}\left( t^{n},x^{n}\right)
=\left( u-\phi \right) \left( t^{n},x^{n}\right) -\underset{=\alpha _{n}}{%
\underbrace{\frac{\alpha _{n}^{2}}{T-t^{n}}}}
\end{equation*}%
and hence $l\geq \left( u-\phi \right) \left( T,\bar{x}\right) $. Since $%
\left( T,\bar{x}\right) $ was a strict maximum point for $u-\phi $ conclude
that $\left( \tilde{t},\tilde{x}\right) =\left( T,\bar{x}\right) $ is the
common limit of the sequences $\left( t^{n},x^{n}\right) ,\left(
t_{n},x_{n}\right) $. Now we note that%
\begin{equation*}
\left( u-\phi \right) \left( t_{n},x_{n}\right) \geq \psi _{n}\left(
t_{n},x_{n}\right) \geq \left( u-\phi \right) \left( t^{n},x^{n}\right)
-\alpha _{n}
\end{equation*}%
which implies that ($o\left( 1\right) \rightarrow 0$ as $n\rightarrow \infty 
$)%
\begin{equation*}
u\left( t_{n},x_{n}\right) \geq u\left( t^{n},x^{n}\right) +o\left( 1\right)
\end{equation*}%
By definition of a subsolution,%
\begin{equation*}
\partial _{t}\phi \left( t_{n},x_{n}\right) -F\left( t_{n},x_{n},u\left(
t_{n},x_{n}\right) ,D\phi \left( t_{n},x_{n}\right) ,D^{2}\phi \left(
t_{n},x_{n}\right) \right) \leq 0
\end{equation*}%
and hence, using properness of $F=F\left( u\right) $, omitting the other
arguments, "with $u=u\left( t_{n},x_{n}\right) $ and $v=u\left(
t^{n},x^{n}\right) +o\left( 1\right) $"; 
\begin{eqnarray*}
-F(u\left( t_{n},x_{n}\right) ) &\geq &-F\left( u\left( t^{n},x^{n}\right)
+o\left( 1\right) \right) +\gamma \left( u\left( t_{n},x_{n}\right) -\left(
u\left( t^{n},x^{n}\right) +o\left( 1\right) \right) \right) \\
&\geq &-F\left( u\left( t^{n},x^{n}\right) \right) +o\left( 1\right) ,
\end{eqnarray*}%
also using uniform continuity of $F$ as function of $u$ over compacts, we
obtain%
\begin{equation*}
\partial _{t}\phi \left( t_{n},x_{n}\right) -F\left( t_{n},x_{n},u\left(
t^{n},x^{n}\right) ,D\phi \left( t_{n},x_{n}\right) ,D^{2}\phi \left(
t_{n},x_{n}\right) \right) \leq o\left( 1\right) .
\end{equation*}%
Sending $n\rightarrow \infty $ yields (use continuity of $\phi $ and $F$)%
\begin{equation*}
\partial _{t}\phi \left( T,\bar{x}\right) -F\left( T,\bar{x},u\left( T,\bar{x%
}\right) ,D\phi \left( T,\bar{x}\right) ,D^{2}\phi \left( T,\bar{x}\right)
\right) \leq 0,
\end{equation*}%
as desired.

\section{Parabolic Comparison: Proof of (ii)}

\begin{proof}
By assumption, $u\left( t,x\right) -v\left( t,y\right) $ is bounded on $%
[0,T]\times \mathbb{R}^{n}\times \mathbb{R}^{n}$ . Let $\left( \hat{t},\hat{x%
},\hat{y}\right) $ be a maximum point of%
\begin{equation}
\phi \left( t,x,y\right) :=u\left( t,x\right) -v\left( t,y\right) -\frac{%
\alpha }{2}\left\vert x-y\right\vert ^{2}-\varepsilon \left( \left\vert
x\right\vert ^{2}+\left\vert y\right\vert ^{2}\right)  \label{phiDef}
\end{equation}%
over $[0,T]\times \mathbb{R}^{n}\times \mathbb{R}^{n}$ where $\alpha >0$ and 
$\varepsilon >0$; such a maximum exists since $\phi \in $ $\mathrm{USC}%
\left( [0,T]\times \mathbb{R}^{n}\times \mathbb{R}^{n}\right) $ and $\phi
\rightarrow -\infty $ as $\left\vert x\right\vert ,\left\vert y\right\vert
\rightarrow \infty $. (The presence $\varepsilon >0$ amounts to a barrier at 
$\infty $ in space ). The plan is to show a "key estimate" of the form%
\newline
\begin{equation}
u\left( t,x\right) -v\left( t,y\right) \leq \inf_{\alpha }\left[ \frac{%
\alpha }{2}\left\vert x-y\right\vert ^{2}+l\left( \alpha \right) \right] ,
\label{KeyEstimate1}
\end{equation}%
valid on $[0,T]\times \mathbb{R}^{n}\times \mathbb{R}^{n}$, where $l\left(
\alpha \right) $ tends to $0$ as $\alpha \uparrow \infty $. Thanks to the
very definition of $\left( \hat{t},\hat{x},\hat{y}\right) $ as $\arg \max $
of $\phi \left( t,x,y\right) =u\left( t,x\right) -v\left( t,y\right) -\frac{%
\alpha }{2}\left\vert x-y\right\vert ^{2}-\varepsilon \left( \left\vert
x\right\vert ^{2}+\left\vert y\right\vert ^{2}\right) $, we obtain the
estimate%
\begin{equation*}
u\left( t,x\right) -v\left( t,y\right) \leq \frac{\alpha }{2}\left\vert
x-y\right\vert ^{2}+\varepsilon \left( \left\vert x\right\vert
^{2}+\left\vert y\right\vert ^{2}\right) +\phi \left( \hat{t},\hat{x},\hat{y}%
\right) .
\end{equation*}%
Note that $\left( \hat{t},\hat{x},\hat{y}\right) $ depends on $\alpha
,\varepsilon $. We shall consider the cases $\hat{t}=0$ and $\hat{t}\in
(0,T] $ separately. In the first case $\hat{t}=0$ we have%
\begin{equation*}
\phi \left( 0,\hat{x},\hat{y}\right) =\sup_{x,y}\left[ u_{0}\left( x\right)
-v_{0}\left( y\right) -\frac{\alpha }{2}\left\vert x-y\right\vert
^{2}-\varepsilon \left( \left\vert x\right\vert ^{2}+\left\vert y\right\vert
^{2}\right) \right] =:A_{\alpha ,\varepsilon }
\end{equation*}%
and lemma \ref{LemmaPenalityOffDiag1} below asserts that $A_{\alpha
,\varepsilon }\rightarrow \sup_{x}\left[ u_{0}\left( x\right) -v_{0}\left(
x\right) \right] \leq 0$ as $\left( \varepsilon ,\alpha \right) \rightarrow
\left( 0,\infty \right) $. The second case is $\hat{t}\in (0,T)$ and we will
show%
\begin{equation}
\phi \left( \hat{t},\hat{x},\hat{y}\right) \leq B_{\alpha ,\varepsilon }%
\text{ where }\left( \text{\ }\limsup_{\varepsilon \rightarrow 0}B_{\alpha
,\varepsilon }\right) \rightarrow 0\text{ as }\alpha \rightarrow \infty ;
\label{PhiLTBae}
\end{equation}%
it is here that we will use theorem of sums and viscosity properites. (Since%
\begin{equation*}
\phi \left( \hat{t},\hat{x},\hat{y}\right) \leq u\left( \hat{t},\hat{x}%
\right) -v\left( \hat{t},\hat{y}\right)
\end{equation*}%
we can and will use the fact that it is enough to consider the case $u\left( 
\hat{t},\hat{x}\right) -v\left( \hat{t},\hat{y}\right) \geq 0$.) Leaving the
details of this to below, let us quickly complete the argument: our
discussion of the two cases above gives $\phi \left( \hat{t},\hat{x},\hat{y}%
\right) \leq A_{\alpha ,\varepsilon }\vee B_{\alpha ,\varepsilon }$ and hence%
\begin{equation*}
u\left( t,x\right) -v\left( t,y\right) \leq \frac{\alpha }{2}\left\vert
x-y\right\vert ^{2}+\varepsilon \left( \left\vert x\right\vert
^{2}+\left\vert y\right\vert ^{2}\right) +A_{\alpha ,\varepsilon }\vee
B_{\alpha ,\varepsilon };
\end{equation*}%
we emphasize that this estimate is valid for all $t,x,y\in \lbrack
0,T]\times \mathbb{R}^{n}\times \mathbb{R}^{n}$ and $\alpha ,\varepsilon >0$%
. Take now $\limsup_{\varepsilon \rightarrow 0}$ on the right hand side,
then optimize over $\alpha >0$, to obtain the key estimate%
\begin{equation*}
u\left( t,x\right) -v\left( t,y\right) \leq \inf_{\alpha }\left\{ \frac{%
\alpha }{2}\left\vert x-y\right\vert ^{2}+l\left( \alpha \right) \right\}
\end{equation*}%
where we may take%
\begin{equation*}
l\left( \alpha \right) :=\limsup_{\varepsilon \rightarrow 0}A_{\alpha
,\varepsilon }\vee \limsup_{\varepsilon \rightarrow 0}B_{\alpha ,\varepsilon
},
\end{equation*}%
noting that $l\left( \alpha \right) $ indeed tends to $0$ as $\alpha
\rightarrow \infty $. It remains to prove the estimate (\ref{PhiLTBae}).$\,$%
To this end, rewrite $\phi $ as 
\begin{equation*}
\phi \left( t,x,y\right) =u^{\varepsilon }\left( t,x\right) -v^{\varepsilon
}\left( t,y\right) -\frac{\alpha }{2}\left\vert x-y\right\vert ^{2}
\end{equation*}%
where $u^{\varepsilon }\left( t,x\right) =u\left( t,x\right) -\varepsilon
\left\vert x\right\vert ^{2}$ and $v^{\varepsilon }\left( t,y\right)
=v\left( t,y\right) +\varepsilon \left\vert y\right\vert ^{2}$. Since $%
u^{\varepsilon }$ (resp. $v^{\varepsilon }$) are upper (resp. lower)
semi-continuous we can apply the (parabolic) theorem of sums as given in the
appendix at $\left( \hat{t},\hat{x},\hat{y}\right) \in (0,T]\times \mathbb{R}%
^{n}\times \mathbb{R}^{n}$ to learn that there are numbers $a,b$ and $X,Y\in
S^{n}$ such that%
\begin{equation}
\left( a,\alpha \left( \hat{x}-\hat{y}\right) ,X\right) \in \mathcal{\bar{P}}%
^{2,+}u^{\varepsilon }\left( \hat{t},\hat{x}\right) ,\,\,\,\left( b,\alpha
\left( \hat{x}-\hat{y}\right) ,Y\right) \in \mathcal{\bar{P}}%
^{2,-}v^{\varepsilon }\left( \hat{t},\hat{y}\right)  \label{SemijetsForUVeps}
\end{equation}%
such that $a-b\geq 0$ (equality if $\hat{t}\in (0,T)$, although this does
not matter), and such that one has the two-sided matrix estimate (\ref%
{MatrixInequality}). It is easy to see (cf. \cite[Remark 2.7]%
{MR1118699UserGuide}) that (\ref{SemijetsForUVeps}) is equivalent to%
\begin{eqnarray*}
\text{ }\left( a,\alpha \left( \hat{x}-\hat{y}\right) +2\varepsilon \hat{x}%
,X+2\varepsilon I\right) &\in &\mathcal{\bar{P}}^{2,+}u\left( \hat{t},\hat{x}%
\right) , \\
\left( b,\alpha \left( \hat{x}-\hat{y}\right) -2\varepsilon \hat{y}%
,Y-2\varepsilon I\right) &\in &\mathcal{\bar{P}}^{2,-}v\left( \hat{t},\hat{y}%
\right) .
\end{eqnarray*}%
Using the viscosity sub- and super-solution properties (and part (i) in the
case that $\hat{t}=T$) we then see that%
\begin{eqnarray*}
a-F\left( \hat{t},\hat{x},u\left( \hat{t},\hat{x}\right) ,\alpha \left( \hat{%
x}-\hat{y}\right) +2\varepsilon \hat{x},X+2\varepsilon I\right) &\leq &0, \\
b-F\left( \hat{t},\hat{y},v\left( \hat{t},\hat{y}\right) ,\alpha \left( \hat{%
x}-\hat{y}\right) -2\varepsilon \hat{y},Y-2\varepsilon I\right) &\geq &0.
\end{eqnarray*}%
Note that (using $a-b\geq 0$)%
\begin{equation}
F\left( \hat{t},\hat{y},v\left( \hat{t},\hat{y}\right) ,\alpha \left( \hat{x}%
-\hat{y}\right) -2\varepsilon \hat{y},Y-2\varepsilon I\right) -F\left( \hat{t%
},\hat{x},u\left( \hat{t},\hat{x}\right) ,\alpha \left( \hat{x}-\hat{y}%
\right) +2\varepsilon \hat{x},X+2\varepsilon I\right) \leq 0
\label{FromViscPropLEZero}
\end{equation}%
Trivially, (recall it is enough to consider the case $u\left( \hat{t},\hat{x}%
\right) \geq v\left( \hat{t},\hat{y}\right) $) 
\begin{eqnarray*}
\gamma \phi \left( \hat{t},\hat{x},\hat{y}\right) &\leq &\gamma \left(
u\left( \hat{t},\hat{x}\right) -v\left( \hat{t},\hat{y}\right) \right) \\
&\leq &F\left( \hat{t},\hat{y},v\left( \hat{t},\hat{y}\right) ,\alpha \left( 
\hat{x}-\hat{y}\right) -2\varepsilon \hat{y},Y-2\varepsilon I\right) \\
&&-F\left( \hat{t},\hat{y},u\left( \hat{t},\hat{x}\right) ,\alpha \left( 
\hat{x}-\hat{y}\right) -2\varepsilon \hat{y},Y-2\varepsilon I\right) \\
&\leq &F\left( \hat{t},\hat{y},v\left( \hat{t},\hat{y}\right) ,\alpha \left( 
\hat{x}-\hat{y}\right) -2\varepsilon \hat{y},Y-2\varepsilon I\right) \\
&&-F\left( \hat{t},\hat{x},u\left( \hat{t},\hat{x}\right) ,\alpha \left( 
\hat{x}-\hat{y}\right) +2\varepsilon \hat{x},X+2\varepsilon I\right) \\
&&+F\left( \hat{t},\hat{x},u\left( \hat{t},\hat{x}\right) ,\alpha \left( 
\hat{x}-\hat{y}\right) +2\varepsilon \hat{x},X+2\varepsilon I\right) \\
&&-F\left( \hat{t},\hat{y},u\left( \hat{t},\hat{x}\right) ,\alpha \left( 
\hat{x}-\hat{y}\right) -2\varepsilon \hat{y},Y-2\varepsilon I\right) \\
&\leq &F\left( \hat{t},\hat{x},u\left( \hat{t},\hat{x}\right) ,\alpha \left( 
\hat{x}-\hat{y}\right) +2\varepsilon \hat{x},X+2\varepsilon I\right) \\
&&-F\left( \hat{t},\hat{y},u\left( \hat{t},\hat{x}\right) ,\alpha \left( 
\hat{x}-\hat{y}\right) -2\varepsilon \hat{y},Y-2\varepsilon I\right)
\end{eqnarray*}%
where we used (\ref{FromViscPropLEZero}) in the last estimate. If $%
\varepsilon $ were absent (e.g. set $\varepsilon =0$ throughout) we would
estimate, with $R:=|u|_{\infty }\vee |v|_{\infty }$, 
\begin{equation*}
F\left( \hat{t},\hat{x},u\left( \hat{t},\hat{x}\right) ,\alpha \left( \hat{x}%
-\hat{y}\right) ,X\right) -F\left( \hat{t},\hat{y},u\left( \hat{t},\hat{x}%
\right) ,\alpha \left( \hat{x}-\hat{y}\right) ,Y\right) \leq \theta
_{R}\left( \alpha \left\vert \hat{x}-\hat{y}\right\vert ^{2}+\left\vert \hat{%
x}-\hat{y}\right\vert \right) =:B_{\alpha }
\end{equation*}%
and since $\alpha \left\vert \hat{x}-\hat{y}\right\vert ^{2}+\left\vert \hat{%
x}-\hat{y}\right\vert \leq 2\alpha \left\vert \hat{x}-\hat{y}\right\vert
^{2}+1/\alpha \rightarrow 0$ as $\alpha \rightarrow \infty $, thanks to \cite%
[lemma 3.1]{MR1118699UserGuide}, we see that $B_{\alpha }\rightarrow 0$ with 
$\alpha \rightarrow \infty $, which is enough to conclude. The present case
where $\varepsilon >0$ is essentially reduced to the case $\varepsilon =0$
by adding/subtracting%
\begin{equation*}
F\left( \hat{t},\hat{x},u\left( \hat{t},\hat{x}\right) ,\alpha \left( \hat{x}%
-\hat{y}\right) ,X\right) -F\left( \hat{t},\hat{y},u\left( \hat{t},\hat{x}%
\right) ,\alpha \left( \hat{x}-\hat{y}\right) ,Y\right) ,
\end{equation*}%
but we need some refined properties of $\left( \hat{t},\hat{x},\hat{y}%
\right) $ as collected in lemma \ref{LemmaPenalityOffDiag2}: (a) $p=\alpha
\left( \hat{x}-\hat{y}\right) $ remains, for fixed $\alpha $, bounded as $%
\varepsilon \rightarrow 0$, (b) $2\varepsilon \left\vert \hat{x}\right\vert $
and $2\varepsilon \left\vert \hat{y}\right\vert $ tend to zero as as $%
\varepsilon \rightarrow 0$ for fixed (large enough) $\alpha $; this follows
from the fact, that for $\alpha $ large enough we must have $%
\limsup_{\varepsilon \rightarrow 0}\varepsilon |\hat{x}|^{2}=c_{\alpha
}<\infty $ (after all, $c_{\alpha }$ tends to zero with $\alpha \rightarrow
\infty $) and by rewriting $\limsup_{\varepsilon \rightarrow 0}\varepsilon |%
\hat{x}|\leq \sqrt{c_{\alpha }}\limsup_{\varepsilon \rightarrow 0}\sqrt{%
\varepsilon }=0$, (c) that $\limsup_{\varepsilon \rightarrow 0}\left( \frac{%
\alpha }{2}\left\vert \hat{x}-\hat{y}\right\vert ^{2}+\left\vert \hat{x}-%
\hat{y}\right\vert \right) \rightarrow 0$ as $\alpha \rightarrow \infty $.
We also note that (\ref{MatrixInequality}) implies (d): any matrix norm of $%
X,Y$ is bounded by a constant times $\alpha $, independent of $\varepsilon $%
. We can now return to the estimate of $\phi $ and clearly have%
\begin{equation*}
\phi \left( \hat{t},\hat{x},\hat{y}\right) \leq \frac{1}{\gamma }\left[
\left( i\right) +\left( ii\right) +\left( ii\right) \right] =:B_{\alpha
,\varepsilon }
\end{equation*}%
where 
\begin{eqnarray*}
\left( i\right) &=&\left\vert F\left( \hat{t},\hat{x},u\left( \hat{t},\hat{x}%
\right) ,\alpha \left( \hat{x}-\hat{y}\right) +2\varepsilon \hat{x}%
,X+2\varepsilon I\right) -F\left( \hat{t},\hat{x},u\left( \hat{t},\hat{x}%
\right) ,\alpha \left( \hat{x}-\hat{y}\right) ,X\right) \right\vert \\
\left( ii\right) &=&\left\vert F\left( \hat{t},\hat{y},u\left( \hat{t},\hat{x%
}\right) ,\alpha \left( \hat{x}-\hat{y}\right) -2\varepsilon \hat{y}%
,Y-2\varepsilon I\right) -F\left( \hat{t},\hat{y},u\left( \hat{t},\hat{x}%
\right) ,\alpha \left( \hat{x}-\hat{y}\right) ,Y\right) \right\vert \\
\left( iii\right) &=&\theta _{R}\left( \alpha \left\vert \hat{x}-\hat{y}%
\right\vert ^{2}+\left\vert \hat{x}-\hat{y}\right\vert \right) \text{.}
\end{eqnarray*}%
From (a),(d) above the gradient and Hessian argument in $F$ as seen in $%
\left( i\right) ,\left( ii\right) $, i.e.%
\begin{equation*}
\alpha \left( \hat{x}-\hat{y}\right) \pm 2\varepsilon \hat{x}\text{ and }%
X+2\varepsilon I,Y-2\varepsilon I,
\end{equation*}%
remain in a bounded set, for fixed $\alpha $, uniformly as $\varepsilon
\rightarrow 0$. From (b) above and the assumed uniform continuity properties
of $F$, it then follows that for fixed (large enough) $\alpha $ 
\begin{equation*}
\left( i\right) ,\left( ii\right) \rightarrow 0\text{ as }\varepsilon
\rightarrow 0.
\end{equation*}%
On the other hand, continuity of $\theta _{R}$ at $0+$ together with (c)
above shows that also $\left( iii\right) \rightarrow 0$ as $\varepsilon <<%
\frac{1}{\alpha }\,\rightarrow 0$. We conclude that%
\begin{equation*}
B_{\alpha ,\varepsilon }\rightarrow 0\text{ as }\varepsilon <<\frac{1}{%
\alpha }\,\rightarrow 0\text{,}
\end{equation*}%
which implies (\ref{PhiLTBae}), as desired. The proof is now finished.
\end{proof}

\begin{lemma}[{\protect\cite[Lemme 2.9]{MR1613876}}]
\bigskip \label{LemmaPenalityOffDiag1}Assume $u_{0},v_{0}\in \mathrm{BUC}%
\left( \mathbb{R}^{n}\right) $. Then%
\begin{equation*}
\sup_{x,y}\left[ u_{0}\left( x\right) -v_{0}\left( y\right) -\frac{\alpha }{2%
}\left\vert x-y\right\vert ^{2}-\varepsilon \left( \left\vert x\right\vert
^{2}+\left\vert y\right\vert ^{2}\right) \right] \rightarrow \sup_{x}\left[
u_{0}\left( x\right) -v_{0}\left( x\right) \right] \text{ as }(\varepsilon ,%
\frac{1}{\alpha })\rightarrow \left( 0,0\right) \text{ .}
\end{equation*}
\end{lemma}

\begin{proof}
Without loss of generality $M:=\sup_{x}\left[ u_{0}\left( x\right)
-v_{0}\left( x\right) \right] >0$; for otherwise replace $u_{0}$ by $%
u_{0}+2\left\vert M\right\vert $. Write $M_{\alpha ,\varepsilon }$ for the
achieved maximum (at $\hat{x},\hat{y}$, say) of the left-hand-side.
Obviously, $u_{0}\left( x\right) -v_{0}\left( x\right) -2\varepsilon
\left\vert x\right\vert ^{2}\leq M_{\alpha ,\varepsilon }$ for any $x$ and so%
\begin{equation*}
M\leq \lim \inf_{\substack{ \varepsilon \rightarrow 0  \\ \alpha \rightarrow
\infty }}M_{\alpha ,\varepsilon }.
\end{equation*}%
(It follows that we can and will consider $\varepsilon \,\left( \alpha
\right) $ small (large) enough so that $M_{\alpha ,\varepsilon }>0$.) On the
other hand, $\left\vert u_{0}\right\vert ,\left\vert v_{0}\right\vert \leq
R<\infty $ and so%
\begin{equation*}
0\leq M_{\alpha ,\varepsilon }\leq 2R-\frac{\alpha }{2}\left\vert \hat{x}-%
\hat{y}\right\vert ^{2}-\varepsilon \left( \left\vert \hat{x}\right\vert
^{2}+\left\vert \hat{y}\right\vert ^{2}\right)
\end{equation*}%
from which we deduce $\frac{\alpha }{2}\left\vert \hat{x}-\hat{y}\right\vert
^{2}\leq 2R$, or $\left\vert \hat{x}-\hat{y}\right\vert \leq \sqrt{4R/\alpha 
}$. By omitting the (positive) penaltity terms, we can also estimate%
\begin{eqnarray*}
M_{\alpha ,\varepsilon } &\leq &u_{0}\left( \hat{x}\right) -v_{0}\left( \hat{%
y}\right) \\
&\leq &u_{0}\left( \hat{x}\right) -v_{0}\left( \hat{x}\right) +\sigma
_{v_{0}}\left( \sqrt{4R/\alpha }\right) \\
&\leq &M+\sigma _{v_{0}}\left( \sqrt{4R/\alpha }\right)
\end{eqnarray*}%
where $\sigma _{v_{0}}$ denotes the modulus of continuity of $v_{0}$. It
follows that%
\begin{equation*}
\lim \sup_{_{\substack{ \varepsilon \rightarrow 0  \\ \alpha \rightarrow
\infty }}}M_{\alpha ,\varepsilon }\leq M
\end{equation*}%
which shows that the $\lim M_{\alpha ,\varepsilon }$ (as $\varepsilon
\rightarrow 0,\alpha \rightarrow \infty $) exists and is equal to $M$.
\end{proof}

\begin{lemma}
\label{LemmaPenalityOffDiag2}Let \bigskip $u\in \mathrm{bUSC}\left(
[0,T]\times \mathbb{R}^{n}\right) $ and $v\in \mathrm{bLSC}\left(
[0,T]\times \mathbb{R}^{n}\right) $. Consider a maximum point $\left( \hat{t}%
,\hat{x},\hat{y}\right) \in (0,T]\times \mathbb{R}^{n}\times \mathbb{R}^{n}$
of%
\begin{equation*}
\phi \left( t,x,y\right) =u\left( t,x\right) -v\left( t,y\right) -\frac{%
\alpha }{2}\left\vert x-y\right\vert ^{2}-\varepsilon \left( \left\vert
x\right\vert ^{2}+\left\vert y\right\vert ^{2}\right) .
\end{equation*}%
where $\alpha ,\varepsilon >0$. Then%
\begin{eqnarray}
\lim \sup_{\varepsilon \rightarrow 0}\alpha \left( \hat{x}-\hat{y}\right)
&=&C\left( \alpha \right) <\infty ,  \label{LemmaAppEst1} \\
\limsup_{\alpha \rightarrow \infty }\limsup_{\varepsilon \rightarrow
0}\varepsilon \left( \left\vert \hat{x}\right\vert ^{2}+\left\vert \hat{y}%
\right\vert ^{2}\right) &=&0,  \label{LemmaAppEst2} \\
\limsup_{\alpha \rightarrow \infty }\limsup_{\varepsilon \rightarrow
0}\left( \frac{\alpha }{2}\left\vert \hat{x}-\hat{y}\right\vert
^{2}+\left\vert \hat{x}-\hat{y}\right\vert \right) &=&0.
\label{LemmaAppEst3}
\end{eqnarray}
\end{lemma}

\begin{remark}
A similar lemma (without $t$ dependence) is found in Barles' book \cite[%
Lemme 4.3]{MR1613876}; the order in which limits are taken is important and
suggests the notation 
\begin{equation*}
\limsup_{\varepsilon <<\frac{1}{\alpha }\,\rightarrow 0}\,\left( ...\right)
:=\limsup_{\alpha \rightarrow \infty }\limsup_{\varepsilon \rightarrow
0}\,\left( ...\right) ,\,\,\liminf_{\varepsilon <<\frac{1}{\alpha }%
\,\rightarrow 0}\,\left( ...\right) :=\liminf_{\alpha \rightarrow \infty
}\liminf_{\varepsilon \rightarrow 0}\,\left( ...\right) .
\end{equation*}
\end{remark}

\begin{proof}
We start with some notation, where unless otherwise stated $t\in \left[ 0,T%
\right] $ and $x,y\in \mathbb{R}^{n}$,%
\begin{eqnarray*}
M_{\alpha ,\varepsilon } &:&=\sup_{t,x,y}\phi \left( t,x,y\right) =u\left( 
\hat{t},\hat{x}\right) -v\left( \hat{t},\hat{y}\right) -\frac{\alpha }{2}%
\left\vert \hat{x}-\hat{y}\right\vert ^{2}-\varepsilon \left( \left\vert 
\hat{x}\right\vert ^{2}+\left\vert \hat{y}\right\vert ^{2}\right) ; \\
M\left( h\right) &:&=\sup_{t,x,y:\left\vert x-y\right\vert \leq h}\left[
u\left( t,x\right) -v\left( t,y\right) \right] \geq \sup_{t,x}\left[ u\left(
t,x\right) -v\left( t,x\right) \right] \\
M^{\prime } &:&=\,\downarrow \lim_{h\rightarrow 0}M\left( h\right)
\end{eqnarray*}%
(As indicated, $M^{\prime }$ exists as limit of $M\left( h\right) $,
non-increasing in $h$ and bounded from below.)

\textbf{Step 1:} Take $t=x=y=0$ as argument of $\phi \left( t,x,y\right) $.
Since $M_{\alpha ,\varepsilon }=\sup \phi $ we have 
\begin{equation*}
c=u\left( 0,0\right) -v\left( 0,0\right) \leq M_{\alpha ,\varepsilon
}=u\left( \hat{t},\hat{x}\right) -v\left( \hat{t},\hat{y}\right) -\frac{%
\alpha }{2}\left\vert \hat{x}-\hat{y}\right\vert ^{2}-\varepsilon \left(
\left\vert \hat{x}\right\vert ^{2}+\left\vert \hat{y}\right\vert ^{2}\right)
\end{equation*}%
and hence, for a suitable constant $C$ (e.g. $C^{2}:=\sup u+\sup \left(
-v\right) +c$) 
\begin{equation*}
\frac{\alpha }{2}\left\vert \hat{x}-\hat{y}\right\vert ^{2}+\varepsilon
\left( \left\vert \hat{x}\right\vert ^{2}+\left\vert \hat{y}\right\vert
^{2}\right) \leq C^{2}
\end{equation*}%
which implies%
\begin{equation}
\left\vert \hat{x}-\hat{y}\right\vert \leq C\sqrt{2/\alpha }
\label{xhatMinusyhat}
\end{equation}%
and hence $\alpha \left\vert \hat{x}-\hat{y}\right\vert \leq \sqrt{2\alpha }%
C $ which is the first claimed estimate (\ref{LemmaAppEst1}).

\textbf{Step 2:} We first argue that it is enough to show the (two)
estimates 
\begin{equation}
\limsup_{\varepsilon <<\frac{1}{\alpha }\,\rightarrow 0}\left[ \,u\left( 
\hat{t},\hat{x}\right) -v\left( \hat{t},\hat{y}\right) \right] \leq
M^{\prime }\leq \liminf_{\varepsilon <<\frac{1}{\alpha }\,\rightarrow
0}\,M_{\alpha ,\varepsilon }.  \label{LemmaAppLeftToDo}
\end{equation}%
Indeed, from $\frac{\alpha }{2}\left\vert \hat{x}-\hat{y}\right\vert
^{2}+\varepsilon \left( \left\vert \hat{x}\right\vert ^{2}+\left\vert \hat{y}%
\right\vert ^{2}\right) =u\left( \hat{t},\hat{x}\right) -v\left( \hat{t},%
\hat{y}\right) -M_{\alpha ,\varepsilon }$ it readily follows that 
\begin{eqnarray*}
\limsup_{\varepsilon <<\frac{1}{\alpha }\,\rightarrow 0}\frac{\alpha }{2}%
\left\vert \hat{x}-\hat{y}\right\vert ^{2}+\varepsilon \left( \left\vert 
\hat{x}\right\vert ^{2}+\left\vert \hat{y}\right\vert ^{2}\right) &\leq
&\limsup_{\varepsilon <<\frac{1}{\alpha }\,\rightarrow 0}\left[ u\left( \hat{%
t},\hat{x}\right) -v\left( \hat{t},\hat{y}\right) -M_{\alpha ,\varepsilon }%
\right] \\
&=&\limsup_{\varepsilon <<\frac{1}{\alpha }\,\rightarrow 0}\left[ \,u\left( 
\hat{t},\hat{x}\right) -v\left( \hat{t},\hat{y}\right) \right]
-\liminf_{\varepsilon <<\frac{1}{\alpha }\,\rightarrow 0}\,M_{\alpha
,\varepsilon } \\
&\leq &0\text{ (and hence }=0\text{).}
\end{eqnarray*}%
This already gives (\ref{LemmaAppEst2}) and also (\ref{LemmaAppEst3}),
noting that%
\begin{equation*}
\left\vert \hat{x}-\hat{y}\right\vert =\alpha ^{-1/2}\alpha ^{1/2}\left\vert 
\hat{x}-\hat{y}\right\vert \leq \frac{1}{2\alpha }+\frac{\alpha }{2}%
\left\vert \hat{x}-\hat{y}\right\vert ^{2}.
\end{equation*}%
We are left to show (\ref{LemmaAppLeftToDo}). For the first estimate, it
suffices to note that, from (\ref{xhatMinusyhat}) and the definition of $%
M\left( h\right) $ applied with $h=C\sqrt{2/\alpha }$,%
\begin{eqnarray*}
\limsup_{\varepsilon <<\frac{1}{\alpha }\,\rightarrow 0}\left[ \,u\left( 
\hat{t},\hat{x}\right) -v\left( \hat{t},\hat{y}\right) \right] &\leq
&\limsup_{\varepsilon <<\frac{1}{\alpha }\,\rightarrow 0}M\left( \sqrt{\frac{%
2}{\alpha }}C\right) \\
&=&\lim_{\alpha \rightarrow \infty }M\left( \sqrt{\frac{2}{\alpha }}C\right)
=M^{\prime }.
\end{eqnarray*}%
We now turn to the second estimate in (\ref{LemmaAppLeftToDo}). From the
very definition of $M^{\prime }$ as $\lim_{h\rightarrow 0}M\left( h\right) $%
, there exists a family $\left( t_{h},x_{h},y_{h}\right) $ so that 
\begin{equation}
|x_{h}-y_{h}|\,\leq h\text{ and }u(t_{h},x_{h})-v(t_{h},x_{h})\rightarrow
M^{\prime }\text{ as }h\rightarrow 0  \label{eqBoundOnDistance}
\end{equation}%
For every $\alpha ,\varepsilon $ we may take $\left(
t_{h},x_{h},y_{h}\right) $ as argument of $\phi $; since $M_{\alpha
,\varepsilon }=\sup \phi $ we have 
\begin{equation}
u(t_{h},x_{h})-v(t_{h},y_{h})-\frac{\alpha }{2}h^{2}-\varepsilon
(|x_{h}|^{2}+|y_{h}|^{2})\leq M_{\alpha ,\varepsilon }  \label{eqSubseq}
\end{equation}%
Take now $\varepsilon =\varepsilon \left( h\right) \rightarrow 0$ with $%
h\rightarrow 0$; fast enough so that $\varepsilon
(|x_{h}|^{2}+|y_{h}|^{2})\rightarrow 0$; for instance $\varepsilon \left(
h\right) :=$ $h/\left( 1+(|x_{h}|^{2}+|y_{h}|^{2})\right) $ would do. It
follows that%
\begin{eqnarray*}
M^{\prime } &=&\lim_{h\rightarrow 0}u(t_{h},x_{h})-v(t_{h},y_{h}) \\
&=&\liminf_{h\,\rightarrow 0}u(t_{h},x_{h})-v(t_{h},y_{h})-\frac{\alpha }{2}%
h^{2}-\varepsilon (|x_{h}|^{2}+|y_{h}|^{2}) \\
&\leq &\liminf_{h\,\rightarrow 0}M_{\alpha ,\varepsilon
_{h}}=\liminf_{\varepsilon \,\rightarrow 0}M_{\alpha ,\varepsilon }\text{ by
monotonicity of }M_{\alpha ,\varepsilon }\text{ in }\varepsilon .
\end{eqnarray*}%
Since this is valid for every $\alpha $, we also have%
\begin{equation*}
M^{\prime }\leq \liminf_{\alpha \,\rightarrow \infty }\liminf_{\varepsilon
\,\rightarrow 0}M_{\alpha ,\varepsilon }.
\end{equation*}%
This is precisely the second estimate in (\ref{LemmaAppLeftToDo}) and so the
proof is finished.
\end{proof}

\section{$\mathrm{BUC}\left( \left[ 0,T\right] \times \mathbb{R}^{n}\right) $
viscosity solutions}

If $F$ satisfies the above structural condition with the further
strengthening that $F$ is \underline{bounded} whenever $u,p,X$ remain
bounded, then any bounded viscosity solution with $\mathrm{BUC}\left( 
\mathbb{R}^{n}\right) $ initial data is in $\mathrm{BUC}\left( \left[ 0,T%
\right] \times \mathbb{R}^{n}\right) .$ More precisely,

\begin{corollary}
\label{CorBUCts}Assume $F$ satisfies the assumptions of section \ref%
{StructAssF} with assumption \ref{StructAssF}) strengthened to%
\begin{equation}
\forall R>0:F|_{\left[ 0,T\right] \times \mathbb{R}^{n}\times \left[ -R,R%
\right] \times B_{R}\times M_{R}}\text{ is bounded, uniformly continuous}.
\label{StructAssFenhanced}
\end{equation}%
Let $u\in \mathrm{BC}\left( [0,T]\times \mathbb{R}^{n}\right) $ be a
viscosity solution to $\partial _{t}-F=0$ on $(0,T)\times \mathbb{R}^{n}$
with intial data $u_{0}=u\left( 0,\cdot \right) \in \mathrm{BUC}\left( 
\mathbb{R}^{n}\right) $. Then 
\begin{equation*}
u=u\left( t,x\right) \in \mathrm{BUC}\left( [0,T]\times \mathbb{R}%
^{n}\right) \text{.}
\end{equation*}
\end{corollary}

\begin{proof}
We adapt the argument from \cite[Lemma 9.1]{MR2010963}. From theorem \ref%
{ThmMain}, there exists a spatial modulus $m$ for $u\left( t,\cdot \right) $%
, uniform over $t\in \left[ 0,T\right] $. Given $0\leq t_{0}<t\leq T$ and $%
x_{0},x\in \mathbb{R}^{n}$ we now estimate, using the triangle inequality,%
\begin{equation*}
\left\vert u\left( t,x\right) -u\left( t_{0},x_{0}\right) \right\vert \leq
m\left( \left\vert x_{0}-x\right\vert \right) +\left\vert u\left(
t,x_{0}\right) -u\left( t_{0},x_{0}\right) \right\vert .
\end{equation*}%
We shall show that $\left\vert u\left( t,x_{0}\right) -u\left(
t_{0},x_{0}\right) \right\vert $ goes to zero as $t\downarrow t_{0}$, 
\textit{uniformly in }$x_{0}\in \mathbb{R}^{n}$ \textit{and }$t_{0}\in
\lbrack 0,T)$. We will show a little more. Fix $x_{0}\in \mathbb{R}^{n}$ and 
$R\in \left( 0,\infty \right) $; for instance $R=1$ would do (and there is
no need to track dependence in $R$). We claim that for every $\eta >0$ one
can find constants $C=C\left( \eta \right) ,K=K\left( \eta \right) $, 
\textit{not dependent on} $x_{0}$ \textit{and} $t_{0}$, such that, for all $%
x\in B_{R/2}\left( x_{0}\right) $ and $y\in B_{R}\left( x_{0}\right) $ and
all $t\in \lbrack t_{0},T]$%
\begin{equation}
u\left( t,y\right) -u\left( t_{0},x\right) \leq \eta +C\left\vert
y-x\right\vert ^{2}+K\left( t-t_{0}\right) \text{ }  \label{BarlesEqu42}
\end{equation}%
and%
\begin{equation}
u\left( t,y\right) -u\left( t_{0},x\right) \geq -\eta -C\left\vert
y-x\right\vert ^{2}-K\left( t-t_{0}\right) .  \label{BarlesEqu43}
\end{equation}%
(Choosing $x=y=x_{0}$ in these estimates shows that $\left\vert u\left(
t,x_{0}\right) -u\left( t_{0},x_{0}\right) \right\vert \leq \inf \left\{
\eta +K\left( \eta \right) \left( t-t_{0}\right) :\eta >0\right\} $ which
immediately gives the desired uniform continuity in time, uniformly in $%
x_{0} $.) We only prove (\ref{BarlesEqu42}), (\ref{BarlesEqu43}) being
proved in an analogous way. In the sequel, $x$ is fixed in $B_{R/2}\left(
x_{0}\right) $. Rewrite (\ref{BarlesEqu42}) as%
\begin{equation*}
u-\chi \leq 0\text{ on }[t_{0},T]\times B_{R}\left( x_{0}\right)
\end{equation*}%
where $\chi \left( t,y\right) :=u\left( t_{0},x\right) +\eta +C\left\vert
y-x\right\vert ^{2}+K\left( t-t_{0}\right) $. We shall see below we can find 
$C$, the choice of which only depends on $\eta $ (and in a harmless way on $%
\left\vert u\right\vert _{\infty ;\left[ 0,T\right] \times \mathbb{R}^{n}},R$
and $m\left( \cdot \right) $ \textit{but not on }$K$ \textit{and not on }$%
x_{0},t_{0}$), such that $u-\chi \leq 0$ on the parabolic boundary of $%
[t_{0},T]\times B_{R}\left( x_{0}\right) $. The extension to the interior is
then based on the maximum principle. More precisely, we can chose $K$
depending on $\eta $ (and again in a harmless way $\left\vert u\right\vert
_{\infty ;\left[ 0,T\right] \times \mathbb{R}^{n}},R$ and $m\left( \cdot
\right) $) such that $\chi $ is a (smooth) strict supersolution of $\partial
_{t}-F$ on $(t_{0},T)\times B_{R}\left( x_{0}\right) $;%
\begin{equation*}
K-F\left( t,y,\chi (t,y),2C\left( y-x\right) ,2CI\right) >0\text{ on }%
(t_{0},T)\times B_{R}\left( x_{0}\right) \text{.}
\end{equation*}%
Indeed, by properness we have 
\begin{equation*}
K-F\left( t,y,\chi (t,y),2C\left( y-x\right) ,2CI\right) >K-F\left(
t,y,-\left\vert u\right\vert _{\infty },2C\left( y-x\right) ,2CI\right) ;
\end{equation*}%
noting $\left\vert y-x\right\vert \leq 2R$ so that $p:=$ $2C\left(
y-x\right) ,X:=2CI$ remain in a bounded set whose size may depend on $\eta $
through $C$, it then follows by our structual assumption on the
non-linearity \footnote{%
... notably boundedness of $F\left( \cdot ,\cdot ,y,p,X\right) $ when $y,p,X$
remain in a bounded set ...} that we can pick $K=K\left( \eta \right) $
large enough such as to achieve the claimed strict inequality. (Note that
this choice of $K$ is uniformly in $t_{0}$ provided we can find $C$ with the
correct dependences.) Since, on the other hand, $u$ is a viscosity solution
(hence subsolution), it follows from the very definition of a subsolution
that%
\begin{equation*}
K-F\left( \hat{t},\hat{y},\chi (t,y),2C\left( \hat{y}-x\right) ,2CI\right)
\leq 0
\end{equation*}%
whenever $\left( \hat{t},\hat{y}\right) \in (t_{0},T]\times B_{R}\left(
x_{0}\right) $ is a maximum point of $u-\chi $. (Note that $\hat{t}=T$ is
possible here, we then rely on part (i) of theorem \ref{ThmMain}.) This
contradiction shows that maximum points of $u-\chi $ over $[t_{0},T]\times 
\bar{B}_{R}\left( x_{0}\right) $ are necessarily achieved on the parabolic
boundary%
\begin{equation*}
\left( t,y\right) \in \lbrack t_{0},T]\times \partial B_{R}\left(
x_{0}\right) \cup \left\{ t_{0}\right\} \times \bar{B}_{R}\left(
x_{0}\right) .
\end{equation*}%
The remainder of the proof is thus concerned with showing that $u-\chi \leq
0 $ on this parabolic boundary. Consider first the case that $t\in \left[
t_{0},T\right] $ and $\left\vert y-x_{0}\right\vert =R$. Since $x\in
B_{R/2}\left( x_{0}\right) $ we must have $\left\vert y-x\right\vert \geq
R/2 $ and it thus suffices to take $C\geq 8\left\vert u\right\vert _{\infty ;%
\left[ 0,T\right] \times \mathbb{R}^{n}}/R^{2}$ to ensure that%
\begin{equation*}
u\left( t_{0},y\right) \leq u\left( t_{0},x\right) +\eta +C\left\vert
y-x\right\vert ^{2}+K\left( t-t_{0}\right)
\end{equation*}%
for all $t\in \left[ t_{0},T\right] $ and $y\in B_{R}\left( x_{0}\right) $,
and any $\eta ,K\geq 0$. The second case to be considered is $t=t_{0}$ and $%
y\in \bar{B}_{R}\left( x_{0}\right) $. We want to see that for every $\eta $
there exists $C=C\left( \eta \right) $ such that 
\begin{equation*}
u\left( t_{0},y\right) \leq u\left( t_{0},x\right) +\eta +C\left\vert
y-x\right\vert ^{2}\text{ for all }y\in \bar{B}_{R}\left( x_{0}\right) ;
\end{equation*}%
but this follows immediately from the fact (cf. theorem \ref{ThmMain}) that $%
u\left( t_{0},\cdot \right) $ has a spatial modulus $m$. Indeed:\ If there
were $\eta >0$ such that for all $C$ there are points $y_{C}$ so that $%
u\left( t_{0},y\right) >u\left( t_{0},x\right) +\eta +C\left\vert
y-x\right\vert ^{2}$, then $\left\vert y_{C}-x\right\vert ^{2}\leq
2\left\vert u\right\vert _{\infty ;\left[ 0,T\right] \times \mathbb{R}%
^{n}}/C\rightarrow 0$ with $C\rightarrow \infty $ and a contradiction to%
\begin{equation*}
m\left( \left\vert y_{C}-x\right\vert \right) \geq u\left( t_{0},y\right)
-u\left( t_{0},x\right) \geq \eta >0\text{.}
\end{equation*}%
is obtained as soon as $C$ is chosen large enough and this choice depends
only on $\eta ,\left\vert u\right\vert _{\infty ;\left[ 0,T\right] \times 
\mathbb{R}^{n}}$ and $m$. Since all these quantities are independent of $%
t_{0}$, so is our choice of $C$.
\end{proof}

\section{Existence}

At last, we discuss existence via Perron's Method; the only difficulty in
the proof is to produce subsolutions and supersolutions.

\begin{theorem}
Assume $F$ satisfies the assumptions of section \ref{StructAssF} with
assumption \ref{StructAssF}) strenghened to%
\begin{equation*}
\forall R>0:F|_{\left[ 0,T\right] \times \mathbb{R}^{n}\times \left[ -R,R%
\right] \times B_{R}\times M_{R}}\text{ is bounded, uniformly continuous.}
\end{equation*}%
Let $u_{0}\in \mathrm{BUC}\left( \mathbb{R}^{n}\right) $. Then there exists $%
u=u\left( t,x\right) \in \mathrm{BUC}\left( [0,T]\times \mathbb{R}%
^{n}\right) $ such that $u$ is a viscosity solution to the initial value
problem%
\begin{eqnarray*}
\partial _{t}-F &=&0\text{ on }(0,T]\times \mathbb{R}^{n}, \\
u\left( 0,\cdot \right) &=&u_{0}\text{.}
\end{eqnarray*}%
(By Theorem \ref{ThmMain} this solution is unique in the class of bounded
viscosity solutions.)
\end{theorem}

\begin{proof}
\underline{Step 1:} Assume $u_{0}$ is Lipschitz continuous with Lipschitz
constant $L$. Define for $z\in \mathbb{R}^{n},\varepsilon >0$ 
\begin{equation*}
\psi _{\varepsilon ,z}(x):=u_{0}(z)-L\left( |x-z|^{2}+\varepsilon \right)
^{1/2}.
\end{equation*}

We will show that there exists $A_{\varepsilon }\leq 0$ (non-positive, yet
to be chosen) such that 
\begin{equation*}
u_{\varepsilon ,z}(t,x):=A_{\varepsilon }t+\psi _{\varepsilon ,z}(x)
\end{equation*}%
is a (classical) subsolution of $\partial _{t}-F=0$. To this end we first
note that $Du_{\varepsilon ,z}=D\psi _{\varepsilon ,z}$ and $%
D^{2}u_{\varepsilon ,z}=D^{2}\psi _{\varepsilon ,z}$ are bounded by $%
LC_{\varepsilon }$ where $C$ is a constant dependent on $\varepsilon $. We
also note that (for any non-positive choice of $A_{\varepsilon }$) 
\begin{equation*}
u_{\varepsilon ,z}(t,x)\leq u_{\varepsilon ,z}(0,x)=\psi _{\varepsilon
,z}\left( x\right) \leq u_{0}(z)-L|x-z|\leq u_{0}(x),
\end{equation*}%
thanks to $L$-Lipschitzness of $u_{0}$. Since $F=F\left( t,x,u,p,X\right) $
is assumed to be proper, and thus in particular anti-monotone in $u$, we have%
\begin{eqnarray*}
&&\partial _{t}u_{\varepsilon ,z}-F\left( t,x,u_{\varepsilon
,z},Du_{\varepsilon ,z},D^{2}u_{\varepsilon ,z}\right) \\
&=&A_{\varepsilon }-F\left( t,x,u_{\varepsilon ,z},D\psi _{\varepsilon
,z},D^{2}\psi _{\varepsilon ,z}\right) \\
&\leq &A_{\varepsilon }-F\left( t,x,\left\vert u_{0}\right\vert _{\infty
},D\psi _{\varepsilon ,z},D^{2}\psi _{\varepsilon ,z}\right) .
\end{eqnarray*}%
Since $\left\vert u_{0}\right\vert _{\infty }<\infty $ and $\left\vert D\psi
_{\varepsilon ,z}\right\vert ,\left\vert D^{2}\psi _{\varepsilon
,z}\right\vert \leq LC_{\varepsilon }$ we can use the assumed boundedness of 
$F$ over sets where $u,p,X$ remain bounded. In particular, we can pick $%
A_{\varepsilon }$ negative, large enough, such that%
\begin{equation*}
\partial _{t}u_{\varepsilon ,z}-F\left( t,x,u_{\varepsilon
,z},Du_{\varepsilon ,z},D^{2}u_{\varepsilon ,z}\right) \leq \dots \leq 0.
\end{equation*}

We now define the $\sup $ of all these subsolutions, 
\begin{equation*}
\hat{u}(t,x):=\sup_{\varepsilon \in (0,1],z\in \mathbb{R}^{n}}u_{\varepsilon
,z}(t,x)\leq u_{0}\left( x\right) \leq \left\vert u_{0}\right\vert _{\infty
}<\infty ,
\end{equation*}%
and note that%
\begin{equation*}
\hat{u}(0,x)=\sup_{\varepsilon \in (0,1],z\in \mathbb{R}^{n}}\psi
_{\varepsilon ,z}\left( x\right) =\sup_{\varepsilon \in (0,1]}u_{0}\left(
x\right) -L\varepsilon ^{1/2}=u_{0}\left( x\right) .
\end{equation*}

Th upper semicontinuous envelope $\underline{u}(t,x):=\hat{u}^{\ast }$ is
then (cf. Proposition 8.2 in \cite{CPrimer} for instance) also a subsolution
to $\partial _{t}-F=0$.

\underline{Step 2:} We show that $\hat{u}(t,x)$ is continous at $t=0$; this
implies that%
\begin{equation*}
\underline{u}(0,x):=\hat{u}\left( 0,x\right) =u_{0}\left( x\right)
\end{equation*}%
and thus yields a sub-solution with the correct initial data. Let $%
(t^{n},x^{n})\rightarrow (0,x)$. First we show lower semicontinuity, i.e. 
\begin{equation*}
\liminf_{n\rightarrow \infty }\hat{u}(t^{n},x^{n})\geq \hat{u}(0,x).
\end{equation*}%
Let $\delta >0$. Choose $\tilde{\varepsilon},\tilde{z}$ such that 
\begin{equation*}
u_{\tilde{\varepsilon},\tilde{z}}(0,x)\geq \hat{u}(0,x)-\delta .
\end{equation*}%
Let $M$ be a bound for $|Du_{\tilde{\varepsilon},\tilde{z}}|$ (and hence for 
$|D\psi _{\tilde{\varepsilon},\tilde{z}}|$). Choose $N$ such that for $n\geq
N$ 
\begin{equation*}
|t^{n}|,|x^{n}-x|\leq \min \left\{ \frac{\delta }{A_{\tilde{\varepsilon}}},%
\frac{\delta }{M}\right\} .
\end{equation*}%
Then 
\begin{align*}
\hat{u}(t^{n},x^{n})& \geq u_{\tilde{\varepsilon},\tilde{z}}(t^{n},x^{n}) \\
& =u_{\tilde{\varepsilon},\tilde{z}}(t^{n},x^{n})-u_{\tilde{\varepsilon},%
\tilde{z}}(0,x)+u_{\tilde{\varepsilon},\tilde{z}}(0,x) \\
& =A_{\tilde{\varepsilon}}t^{n}+\psi _{\tilde{\varepsilon},\tilde{z}%
}(x^{n})-\psi _{\tilde{\varepsilon},\tilde{z}}(x) + u_{\tilde{\varepsilon},%
\tilde{z}}(0,x) \\
& \geq \hat{u}(0,x)-3\delta ,
\end{align*}%
which proves the lower semicontinuity.

For upper semicontinuity, notice that 
\begin{align*}
u_{\varepsilon ,z}(s,y)& =A_{\varepsilon }s+\psi _{\varepsilon ,z}(y) \\
& \leq A_{\varepsilon }s+u_{0}(y) \\
& \leq u_{0}(y),
\end{align*}%
where we have used that $A_{\varepsilon }\leq 0$ and that $\psi
_{\varepsilon ,z}(y)\leq u_{0}(y)$, as shown above. Hence, $\hat{u}(s,y)\leq
u_{0}(y)$, and then for $(t^{n},x^{n})\rightarrow (0,x)$, we have 
\begin{equation*}
\limsup_{n}\hat{u}(t^{n},x^{n})\leq \limsup_{n}u_{0}(x^{n})=u_{0}(x)=\hat{u}%
(0,x).
\end{equation*}%
Hence $\hat{u}$ is also upper semicontinuous at $(0,x)$ and hence continuous
at $(0,x)$.

\underline{Step 3:} Similarly, one constructs a super-solution with correct
(bounded, Lipschitz) initial data $u_{0}$.\ Perron's method then applies and
yields a bounded viscosity solution to $\partial _{t}-F=0$ with bounded,
Lipschitz initial data.

\underline{Step 4:} Let now $u_{0}\in \mathrm{BUC}(\mathbb{R}^{n})$ and $%
u_{0}^{n}$ be a sequence of bounded Lipschitz functions such that $%
|u_{0}^{n}-u_{0}|_{\infty }\rightarrow 0$. By the previous step there exists
a bounded solution $u^{n}$ to $\partial _{t}-F=0$ with initial data $%
u^{n}\left( 0,\cdot \right) =u_{0}^{n}$. (It is also unique by comparison.)
Since $F$ is proper $\left( \gamma \geq 0\right) $, the solutions form a
contraction in the sense%
\begin{equation*}
|u^{n}-u^{m}|_{\infty ;[0,T]\times \mathbb{R}^{n}}\leq
|u_{0}^{n}-u_{0}^{m}|_{\infty ;\mathbb{R}^{n}}
\end{equation*}%
(This follows immediately from comparison and properness.). Hence $u^{n}$ is
Cauchy in supremum norm and converges to a continuous bounded function $%
u:[0,T]\times \mathbb{R}^{n}\rightarrow \mathbb{R}$. By Lemma 6.1 in the
User's Guide we then have that $u$ is a bounded solution to $\partial
_{t}-F=0$ with $\mathrm{BUC}(\mathbb{R}^{n})$ initial data. By comparison,
it is the unique (bounded) solution with this initial data. At last,
corollary \ref{CorBUCts} shows that the solution is $\mathrm{BUC}$ in time
space.
\end{proof}

\section{Appendix1 :\ Recalls on parabolic jets}

If $u:\left( 0,T\right) \times \mathbb{R}^{n}\rightarrow \mathbb{R}$ its
parabolic semijet $\mathcal{P}^{2,+}u$ is defined by $\left( b,p,X\right)
\in \mathbb{R\times R}^{n}\mathbb{\times }\mathcal{S}^{n}$ lies in $\mathcal{%
P}^{2,+}u\left( s,z\right) $ if $\left( s,z\right) \in \left( 0,T\right)
\times \mathbb{R}^{n}$ and%
\begin{equation*}
u\left( t,x\right) \leq u\left( s,z\right) +b\left( t-s\right) +\left\langle
p,x-z\right\rangle +\frac{1}{2}\left\langle X\left( x-z\right)
,x-z\right\rangle +o\left( \left\vert t-s\right\vert +\left\vert
x-z\right\vert ^{2}\right)
\end{equation*}%
as $\left( 0,T\right) \times \mathbb{R}^{n}\ni \left( t,x\right) \rightarrow
\left( s,z\right) $. Consider now $u:Q\rightarrow \mathbb{R}$ where $%
Q=(0,T]\times \mathbb{R}^{n}$. The parabolic semijet relative to $Q$, write $%
\mathcal{P}_{Q}^{2,+}u$, as used in \cite{MR1119185} for instance, is
defined by $\left( b,p,X\right) \in \mathbb{R\times R}^{n}\mathbb{\times }%
\mathcal{S}^{n}$ lies in $\mathcal{P}_{Q}^{2,+}u\left( s,z\right) $ if $%
\left( s,z\right) \in \left( 0,T\right) \times \mathbb{R}^{n}$ and%
\begin{equation*}
u\left( t,x\right) \leq u\left( s,z\right) +b\left( t-s\right) +\left\langle
p,x-z\right\rangle +\frac{1}{2}\left\langle X\left( x-z\right)
,x-z\right\rangle +o\left( \left\vert t-s\right\vert +\left\vert
x-z\right\vert ^{2}\right)
\end{equation*}%
as $Q\ni \left( t,x\right) \rightarrow \left( s,z\right) $. Note that $%
\mathcal{P}_{Q}^{2,+}u\left( s,z\right) =\mathcal{P}^{2,+}u\left( s,z\right) 
$ for $\left( s,z\right) \in \left( 0,T\right) \times \mathbb{R}^{n}$. Note
also the special behaviour of the semijet at time $T$ in the sense that%
\begin{equation}
\left( b,p,X\right) \in \mathcal{P}_{Q}^{2,+}u\left( T,z\right) \implies
\forall b^{\prime }\leq b:\left( b^{\prime },p,X\right) \in \mathcal{P}%
_{Q}^{2,+}u\left( T,z\right) \text{.}  \label{SpecialPropPQ}
\end{equation}%
Closures of these jets are defined in the usual way; e.g.%
\begin{equation*}
\left( b,p,X\right) \in \mathcal{\bar{P}}_{Q}^{2,+}u\left( T,z\right)
\end{equation*}%
iff $\exists \left( t_{n},z_{n};b_{n},p_{n},X_{n}\right) \in Q\times \mathbb{%
R\times R}^{n}\mathbb{\times }\mathcal{S}^{n}:\left(
b_{n},p_{n},X_{n}\right) \in \mathcal{\bar{P}}_{Q}^{2,+}u\left(
t_{n},z_{n}\right) $ and%
\begin{equation*}
\left( t_{n},z_{n};u\left( t_{n},z_{n}\right) ;b_{n},p_{n},X_{n}\right)
\rightarrow \left( T,z;u\left( T,z\right) ;b,p,X\right) .
\end{equation*}

\section{Appendix 2: parabolic theorem of sums revisited}

\begin{theorem}[{\protect\cite[Thm 7]{MR1073054}}]
Let $u_{1},u_{2}\in \mathrm{USC}\left( (0,T)\times \mathbb{R}^{n}\right) $
and $w\in \mathrm{USC}\left( (0,T)\times \mathbb{R}^{2n}\right) $ be given by%
\begin{equation*}
w\left( t,x\right) =u_{1}\left( t,x_{1}\right) +u_{2}\left( t,x_{2}\right)
\end{equation*}%
Suppose that $s\in \left( 0,T\right) ,\,z=\left( z_{1},z_{2}\right) \in 
\mathbb{R}^{2n},\,b\in \mathbb{R},\,p=\left( p_{1},p_{2}\right) \in \mathbb{R%
}^{2n},A\in \mathcal{S}^{2n}$ with%
\begin{equation}
\left( b,p,A\right) \in \mathcal{P}^{2,+}w\left( s,z\right) .
\label{bpAinJet}
\end{equation}%
Assume moreover that there is an $r>0$ such that for every $M>0$ there is a $%
C$ such that for $i=1,2$%
\begin{eqnarray}
b_{i} &\leq &C\text{ whenever }\left( b_{i},q_{i},X_{i}\right) \in \mathcal{P%
}^{2,+}w\left( t,x_{i}\right) ,  \label{CI_Equ27} \\
\left\vert x_{i}-z_{i}\right\vert +\left\vert s-t\right\vert &<&r\text{ and }%
\left\vert u_{i}\left( t,x_{i}\right) \right\vert +\left\vert
q_{i}\right\vert +\left\Vert X_{i}\right\Vert \leq M.  \notag
\end{eqnarray}%
Then for each $\varepsilon >0$ there exists $\left( b_{i},\,X_{i}\right) \in 
\mathbb{R\times }\mathcal{S}^{n}$ such that%
\begin{equation*}
\left( b_{i},p_{i},X_{i}\right) \in \mathcal{\bar{P}}^{2,+}u\left(
s,z_{i}\right)
\end{equation*}%
and%
\begin{equation}
-\left( \frac{1}{\varepsilon }+\left\Vert A\right\Vert \right) I\leq \left( 
\begin{array}{cc}
X_{1} & 0 \\ 
0 & X_{2}%
\end{array}%
\right) \leq A+\varepsilon A^{2}\text{ and }b_{1}+b_{2}=b.  \label{CI_Equ28}
\end{equation}
\end{theorem}

The proof of the above theorem is reduced (cf. Lemma 8 in \cite{MR1073054})
to the case $b=0,z=0,p=0$ and $v_{1}\left( s,0\right) =v_{2}\left(
s,0\right) =0$, where (in order to avoid confusion) we write $v_{i}$ instead
of $u_{i}$. Condition (\ref{bpAinJet}) translates than to%
\begin{equation}
v_{1}\left( t,x_{1}\right) +v_{2}\left( t,x_{2}\right) -\frac{1}{2}%
\left\langle Ax,x\right\rangle \leq 0\text{ for all }\left( t,x\right) \in
\left( 0,T\right) \times \mathbb{R}^{2n};  \label{CI_Equ29}
\end{equation}%
this also means that the left-hand-side as a function of $\left(
t,x_{1},x_{2}\right) $ has a global maximum at $\left( s,0,0\right) $. The
assertion of the (reduced) theorem is then the existence of $\left(
b_{i},\,X_{i}\right) \in \mathbb{R\times }\mathcal{S}^{n}$ such that $\left(
b_{i},0,X_{i}\right) \in \mathcal{\bar{P}}^{2,+}v_{i}\left( s,0\right) $ for 
$i=1,2$ and (\ref{CI_Equ28}) holds with $b=0$.

\begin{theorem}
\label{ThmTOSrevisited}Assume that $u_{i}$ has a finite extension to $%
(0,T]\times \mathbb{R}^{n},\,i=1,2$, via its semi-continuous envelopes, that
is,%
\begin{equation*}
u_{i}\left( T,x\right) =\limsup_{\substack{ \left( t,y\right) \in
(0,T)\times \mathbb{R}^{n}:  \\ t\uparrow T,y\rightarrow x}}u_{i}\left(
t,y\right) <\infty .
\end{equation*}%
Then the above theorem remains valid at $s=T$ if%
\begin{equation*}
\mathcal{P}^{2,+}w\left( s,z\right) \text{ and }\mathcal{\bar{P}}%
^{2,+}u\left( s,z_{i}\right)
\end{equation*}%
is replaced by%
\begin{equation*}
\mathcal{P}_{Q}^{2,+}w\left( T,z\right) \text{ and }\mathcal{\bar{P}}%
_{Q}^{2,+}u\left( T,z_{i}\right)
\end{equation*}%
and the final equality in (\ref{CI_Equ28}) is replaced by 
\begin{equation}
b_{1}+b_{2}\geq b\text{.}  \label{b1b2GTb}
\end{equation}
\end{theorem}

\begin{remark}
If we knew (but we don't!) that the final conclusion is $\left(
b_{i},p_{i},X_{i}\right) \in \mathcal{P}^{2,+}u\left( T,z_{i}\right) $,
rather than just being an element in the closure $\mathcal{\bar{P}}%
_{Q}^{2,+}u\left( T,z_{i}\right) $, then we could trivially diminuish the $%
b_{i}$'s such as to have $b_{1}+b_{2}=b$; cf. (\ref{SpecialPropPQ}).
\end{remark}

\begin{proof}
\textbf{Step 1:}\ We focus on the reduced setting (and thus write $v_{i}$
instead of $u_{i}$) and (following the proof of Lemma 8 in \cite{MR1073054})
redefine $v_{i}\left( t_{i},x_{i}\right) $ as $-\infty $ when $\left\vert
x_{i}\right\vert >1$ or $t_{i}\notin \left[ T/2,T\right] $. We can also
assume that (\ref{CI_Equ29}) is strict if $t<s=T$ or $x\neq 0$. For the rest
of the proof, we shall abreviate $\left( t_{1},t_{2}\right) ,\left(
x_{1},x_{2}\right) $ etc by $\left( t,x\right) $. With this notation in mind
we set%
\begin{equation*}
w\left( t,x\right) =v_{1}\left( t_{1},x_{1}\right) +v_{2}\left(
t_{2},x_{2}\right) -\frac{1}{2}\left\langle Ax,x\right\rangle .
\end{equation*}%
By the extension via semi-continuous envelopes, there exist a sequence $%
\left( t^{n},x^{n}\right) \in \left( 0,T\right) ^{2}\times \left( \mathbb{R}%
^{n}\right) ^{2}$, such that%
\begin{equation*}
\left( t^{n},x^{n}\right) \equiv \left(
t^{1,n},t^{2,n},x^{1,n},x^{2,n}\right) \rightarrow \left( T,T,0,0\right) .
\end{equation*}%
We now consider $w$ with a penality term for $t_{1}\neq t_{2}$ and a barrier
at time $T$ for both $t_{1}$ and $t_{2}$.%
\begin{equation*}
\psi _{m,n}\left( t,x\right) =w\left( t,x\right) -\left\{ \frac{m}{2}%
\left\vert t_{1}-t_{2}\right\vert ^{2}+\sum_{i=1}^{2}\left( T-t^{i,n}\right)
^{2}/\left( T-t_{i}\right) \right\} ,
\end{equation*}%
indexd by $\left( m,n\right) \in \mathbb{N}^{2}$, say. By assumption $w$ has
a maximum at $\left( T,T,0,0\right) $ which we may assume to be strict
(otherwise subtract suitable forth order terms ...). Define now%
\begin{equation*}
\left( \hat{t},\hat{x}\right) \in \arg \max \psi _{m,n}\text{ over }%
[T-r,T]^{2}\times \bar{B}_{r}\left( 0\right) ^{2}
\end{equation*}%
where $r=T/2$ (for instance). When we want to emphasize dependence on $m,n$
we write $\left( \hat{t}_{m,n},\hat{x}_{m,n}\right) $. We shall see below
(Step 2) that there exists increasing sequences $m=m\left( k\right)
,n=n\left( k\right) $ so that%
\begin{equation}
\left( \hat{t},\hat{x}\right) |_{m=m\left( k\right) ,n=n\left( k\right)
}\rightarrow \left( T,T,0,0\right) .  \label{thatxhatTendstoTT00}
\end{equation}%
Using the (elliptic) theorem of sums in the form of \cite[Theorem 1]%
{MR1073054} we find that there are%
\begin{equation*}
\left( b_{i},p_{i},X_{i}\right) \in \mathcal{\bar{P}}^{2,+}v_{i}\left( \hat{t%
}_{i},\hat{x}_{i}\right)
\end{equation*}%
(where $\hat{t}_{i}\rightarrow T,\hat{x}_{i}\rightarrow 0$ as $k\rightarrow
\infty $) such that the first part of (\ref{CI_Equ28}) holds and%
\begin{equation*}
A\left( 
\begin{array}{c}
\hat{x}_{1} \\ 
\hat{x}_{2}%
\end{array}%
\right) =\left( 
\begin{array}{c}
p_{1} \\ 
p_{2}%
\end{array}%
\right) ,\,\,b_{i}=m\left( t_{i}-t_{3-i}\right) +\left( T-t^{i,\varepsilon
}\right) ^{2}/\left( T-t_{i}\right) ^{2}.
\end{equation*}%
for $i=1,2$. Note that%
\begin{equation*}
b_{1}+b_{2}=m\left( t_{1}-t_{2}\right) +m\left( t_{2}-t_{1}\right) +\left( 
\text{positive terms}\right) \geq 0;
\end{equation*}%
since each $b_{i}$ is bounded above by the assumptions and the estimates on
the $X_{i}$ it follows that the $b_{i}$ lie in precompact sets. Upon passing
to the limit $k\rightarrow \infty $ we obtain points%
\begin{equation*}
\left( b_{i},p_{i},X_{i}\right) \in \mathcal{\bar{P}}^{2,+}v_{i}\left(
T,0\right) ,\,\,\,i=1,2;
\end{equation*}%
with $b_{1}+b_{2}\geq 0$.

\textbf{Step 2:}\ We still have to establish (\ref{thatxhatTendstoTT00}). We
first remark that for arbitrary (strictly) increasing sequences $m\left(
k\right) ,n\left( k\right) $, compactness implies that%
\begin{equation*}
\left\{ \left( \hat{t}_{m\left( k\right) ,n\left( k\right) },\hat{x}%
_{m\left( k\right) ,n\left( k\right) }\right) :k\geq 1\right\} \in \lbrack
T-r,T]^{2}\times \bar{B}_{r}\left( 0\right) ^{2}
\end{equation*}%
has limit points. Note also $\hat{t}_{1},\hat{t}_{2}\in \lbrack T-r,T)$
thanks to the barrier at time $T$. The key technical ingredient for the
remained of the argument is and we postpone details of these to Step 3 below:%
\begin{equation}
w\left( \hat{t},\hat{x}\right) -\psi _{m,n}\left( \hat{t},\hat{x}\right)
=\left\{ \frac{m}{2}\left\vert \hat{t}_{1}-\hat{t}_{2}\right\vert
^{2}+\sum_{i=1}^{2}\left( T-t^{i,n}\right) ^{2}/\left( T-\hat{t}_{i}\right)
\right\} \rightarrow 0\text{ as }\frac{1}{n}<<\frac{1}{m}\rightarrow 0.
\label{Step3Stmt}
\end{equation}%
In particular, for every $k>0$ there exists $m\left( k\right) $ such that
for all $m\geq m\left( k\right) $ 
\begin{equation*}
\lim \sup_{n\rightarrow \infty }\left\{ ...\right\} <\frac{1}{k}.
\end{equation*}%
By making $m\left( k\right) $ larger if necessary we may assume that $%
m\left( k\right) $ is (strictly)\ increasing in $k$. Furthermore there
exists $n\left( m\left( k\right) ,k\right) =n\left( k\right) $ such that for
all $n\geq n\left( k\right) :\left\{ ...\right\} <2/k$. Again, we may make $%
n\left( k\right) $ larger if necessary so that $n\left( k\right) $ is
strictly increasing. Recall $t^{1,n\left( k\right) }-t^{2,n\left( k\right)
}\rightarrow T-T=0$ as $k\rightarrow \infty $. For reasons that will become
apparent further below, we actually want the stronger statement that 
\begin{equation}
\frac{m\left( k\right) }{2}\left\vert t^{1,n\left( k\right) }-t^{2,n\left(
k\right) }\right\vert ^{2}\rightarrow 0\text{ as }k\rightarrow \infty
\label{nkFastEnough}
\end{equation}%
which we can achieve by modifying $n\left( k\right) $ such as to run to $%
\infty $ even faster. Note that the so-constructed $m=m\left( k\right)
,n=n\left( k\right) $ has the property 
\begin{equation}
\left[ w\left( \hat{t},\hat{x}\right) -\psi _{m,n}\left( \hat{t},\hat{x}%
\right) \right] |_{m=m\left( k\right) ,n=\left( k\right) }=\left\{
...\right\} |_{m=m\left( k\right) ,n=\left( k\right) }\rightarrow 0\text{ as 
}k\rightarrow \infty .  \label{wApproxEqualPsi}
\end{equation}%
By switching to a subsequence $\left( k_{l}\right) $ if necessary we may
also assume (after relabeling) that%
\begin{equation*}
\left( \hat{t}_{m\left( k\right) ,n\left( k\right) },\hat{x}_{m\left(
k\right) ,n\left( k\right) }\right) \rightarrow \left( \tilde{t},\tilde{x}%
\right) \in \lbrack T-r,T]^{2}\times \bar{B}_{r}\left( 0\right) ^{2}\text{
as }k\rightarrow \infty \text{.}
\end{equation*}%
In the sequel we think of $\left( \hat{t},\hat{x}\right) $ as this sequence
indexed by $k$. We have%
\begin{eqnarray}
w\left( \tilde{t},\tilde{x}\right) &\geq &\lim \sup_{k\rightarrow \infty
}w\left( \hat{t},\hat{x}\right) |_{m=m\left( k\right) ,n=\left( k\right) }%
\text{ \ \ by upper-semi-continuity }  \label{wttildextildeestimate} \\
&=&\lim \sup_{k\rightarrow \infty }\psi _{m,n}\left( \hat{t},\hat{x}\right)
|_{m=m\left( k\right) ,n=\left( k\right) }\text{ thanks to (\ref%
{wApproxEqualPsi}).}  \notag
\end{eqnarray}%
On the other hand, thanks to the particular form of our time-$T$ barrier,%
\begin{eqnarray*}
&&\psi _{m,n}\left( \hat{t},\hat{x}\right) \\
&\geq &\psi _{m,n}\left( t^{n},x^{n}\right) \\
&=&w\left( t^{n},x^{n}\right) -\left\{ \frac{m}{2}\left\vert
t^{1,n}-t^{2,n}\right\vert ^{2}+\sum_{i=1}^{2}\left( T-t^{i,n}\right)
\right\} .
\end{eqnarray*}%
Take now $m=m\left( k\right) ,n=n\left( k\right) $ as constructed above. Then%
\begin{eqnarray*}
&&\psi _{m,n}\left( \hat{t},\hat{x}\right) |_{m=m\left( k\right) ,n=\left(
k\right) } \\
&\geq &w\left( t^{n\left( k\right) },x^{n\left( k\right) }\right) \\
&&-\left\{ \frac{m\left( k\right) }{2}\left\vert t^{1,n\left( k\right)
}-t^{2,n\left( k\right) }\right\vert ^{2}+\sum_{i=1}^{2}\left(
T-t^{i,n\left( k\right) }\right) \right\}
\end{eqnarray*}%
The first term in the curly bracket goes to zero (with $k\rightarrow \infty $%
) thanks to (\ref{nkFastEnough}), the other term goes to zero since $%
t^{i,n}\rightarrow T$ with $n\rightarrow \infty $, and hence also along $%
n\left( k\right) $. On the other hand (recall $x^{i,n}\rightarrow 0$) 
\begin{equation*}
w\left( t^{n\left( k\right) },x^{n\left( k\right) }\right) \rightarrow
v_{1}\left( T,0\right) +v_{2}\left( T,0\right) -\frac{1}{2}\left\langle
A0,0\right\rangle =0\text{ as }k\rightarrow \infty .
\end{equation*}%
(In the reduced setting $v_{1}\left( T,0\right) =v_{2}\left( T,0\right) =0$%
.) It follows that%
\begin{equation*}
\lim \inf_{k\rightarrow \infty }\psi _{m,n}\left( \hat{t},\hat{x}\right)
|_{m=m\left( k\right) ,n=\left( k\right) }=0.
\end{equation*}%
Together with (\ref{wttildextildeestimate}) we see that $w\left( \tilde{t},%
\tilde{x}\right) \geq 0$. But $w\left( T,T,0,0\right) =0$ was a strict
maximum in $[T-r,T]^{2}\times \bar{B}_{r}\left( 0\right) ^{2}$ and so we
must have $\left( \tilde{t},\tilde{x}\right) =\left( T,T,0,0\right) $.

\textbf{Step 3: }Set%
\begin{equation*}
M\left( h\right) =\sup_{\substack{ \left( t,x\right) \in \lbrack
T-r,T)^{2}\times \bar{B}_{r}\left( 0\right) ^{2}  \\ \left\vert
t_{1}-t_{2}\right\vert <h}}w\left( t_{1},t_{2},x_{1},x_{2}\right) \text{ and 
}M^{\prime }=\lim_{h\rightarrow 0}M\left( h\right)
\end{equation*}%
It is enough to show 
\begin{equation}
\limsup_{\frac{1}{n}<<\frac{1}{m}\rightarrow 0}w\left( \hat{t},\hat{x}%
\right) \leq M^{\prime }\leq \liminf_{\frac{1}{n}<<\frac{1}{m}\rightarrow
0}\,\psi _{m,n}\left( \hat{t},\hat{x}\right) .  \label{LemmaAppLeftToDoTOS}
\end{equation}%
since the claimed 
\begin{equation*}
w\left( \hat{t},\hat{x}\right) -\psi _{m,n}\left( \hat{t},\hat{x}\right)
=\left\{ \frac{m}{2}\left\vert \hat{t}_{1}-\hat{t}_{2}\right\vert
^{2}+\sum_{i=1}^{2}\left( T-t^{i,n}\right) ^{2}/\left( T-\hat{t}_{i}\right)
\right\} \rightarrow 0\text{ as }\frac{1}{n}<<\frac{1}{m}\rightarrow 0.
\end{equation*}%
follows from 
\begin{eqnarray*}
\limsup_{\frac{1}{n}<<\frac{1}{m}\rightarrow 0}\left\{ ...\right\} &\leq
&\limsup_{\frac{1}{n}<<\frac{1}{m}\rightarrow 0}w\left( \hat{t},\hat{x}%
\right) -\liminf_{\frac{1}{n}<<\frac{1}{m}\rightarrow 0}\,\psi _{m,n}\left( 
\hat{t},\hat{x}\right) \\
&\leq &0\text{ (and hence }=0\text{).}
\end{eqnarray*}%
Note that $w\left( \hat{t},\hat{x}\right) $ is bounded on $[T-r,T]^{2}\times 
\bar{B}_{r}\left( 0\right) ^{2}$ so that%
\begin{equation*}
\left\vert \hat{t}_{1}-\hat{t}_{2}\right\vert ^{2}=O\left( 1/m\right)
\implies w\left( \hat{t},\hat{x}\right) \leq M\left( \text{const}/\sqrt{m}%
\right) .
\end{equation*}%
On the other hand, from the very definition of $M^{\prime }$ as $%
\lim_{h\rightarrow 0}M\left( h\right) $, there exists a family $\left(
t_{h},x_{h}\right) $ so that 
\begin{equation}
|t_{1,h}-t_{2,h}|\,\leq h\text{ and }w\left( t_{h},x_{h}\right) \rightarrow
M^{\prime }\text{ as }h\rightarrow 0  \label{eqBoundOnDistanceTOS}
\end{equation}%
For every $m,n$ we may take $\left( t_{h},x_{h}\right) $ as argument of $%
\psi _{m,n}$ (which itself has a maximum at $\hat{t},\hat{x}$); hence 
\begin{equation}
w(t_{h},x_{h})-\frac{m}{2}h^{2}-\sum_{i=1}^{2}\left( T-t^{i,n}\right)
^{2}/\left( T-t_{i,h}\right) \leq \psi _{m,n}\left( \hat{t},\hat{x}\right) .
\label{eqSubseqTOS}
\end{equation}%
Take now a sequence $n=n\left( h\right) $, fast enough increasing as $%
h\searrow $ such that $\left( T-t^{i,n}\right) ^{2}/\left( T-t_{i,h}\right)
\rightarrow 0$ with $h\rightarrow 0$. It follows that%
\begin{eqnarray*}
M^{\prime } &=&\lim_{h\rightarrow 0}w(t_{h},x_{h}) \\
&=&\liminf_{h\,\rightarrow 0}\left( w(t_{h},x_{h})-\frac{m}{2}%
h^{2}-\sum_{i=1}^{2}\left( T-t^{i,n\left( h\right) }\right) ^{2}/\left(
T-t_{i,h}\right) \right) \\
&\leq &\liminf_{h\,\rightarrow 0}\psi _{m,n\left( h\right) }\left( \hat{t},%
\hat{x}\right) =\liminf_{n\,\rightarrow \infty }\psi _{m,n}\left( \hat{t},%
\hat{x}\right) \text{ by monotonicity of }\sup \psi _{m,n}\text{ in }n.
\end{eqnarray*}%
(In the last equality we used that $t^{i,n}\uparrow T$; this shows that $%
\sup \psi _{m,n}$ is indeed monoton in $n$.) The proof is now finished.
\end{proof}


\begin{thebibliography}{99}
\bibitem{MR1613876} Guy Barles, 
\newblock {\em Solutions de viscosit\'e des
\'equations de {H}amilton-{J}acobi}. \newblock Springer, 2004.

\bibitem{MR1484411} Martino Bardi and Italo Capuzzo-Dolcetta%
\newblock {\em
Optimal control and viscosity solutions of {H}amilton-{J}acobi-{B}ellman
equations }. \newblock Birkhaeuser, 1997.

\bibitem{MR2010963} Guy Barles, Samuel Biton, Mariane Bourgoing, and Olivier
Ley. \newblock Uniqueness results for quasilinear parabolic equations
through viscosity solutions' methods. 
\newblock {\em Calc. Var. Partial
Differential Equations}, 18(2):159--179, 2003.

\bibitem{MR1904498} Guy Barles, Samuel Biton, Olivier Ley. \newblock A
geometrical approach to the study of unbounded solutions of quasilinear
parabolic equations. 
\newblock {\em Archive for Rational Mechanics and
Analysis}, 162, no. 4, 287--325, 2002.

\bibitem{CFO} M. Caruana, P. Friz and H. Oberhauser: A (rough) pathwise
approach to a class of nonlinear SPDEs, Annales de l'Institut Henri Poincare
/ Analyse non lineaire, ISSN: 0294-1449, DOI: 10.1016/j.anihpc.2010.11.002

\bibitem{CGG91} Yun-Gang Chen, Yoshikazu Giga, and Shun'ichi Goto: Remarks
on viscosity solutions for evolution equations, Proc. Japan Acad. Ser. A
Math. Sci. Volume 67, Number 10 (1991),

\bibitem{CSTV} P. Cheridito, M. Soner, N. Touzi and Nicolas Victoir, Second
Order Backward Stochastic Differential Equations and Fully Non-Linear
Parabolic PDEs. Communications in Pure and Applied Mathematics 60 (7):
1081-1110 (2007)

\bibitem{CPrimer} Crandall, Michael G.; Viscosity solutinos: A Primer.
Lecture notes in Mathematics, Springer. 1996

\bibitem{MR1073054} Crandall, Michael G.; Ishii, Hitoshi. \newblock The
maximum principle for semicontinuous functions.%
\newblock {\em Differential
Integral Equations.} 3 (1990), no. 6, 1001--1014.

\bibitem{MR1118699UserGuide} Michael~G. Crandall, Hitoshi Ishii, and
Pierre-Louis Lions. \newblock User's guide to viscosity solutions of second
order partial differential equations. 
\newblock {\em Bull. Amer. Math. Soc.
(N.S.)}, 27(1):1--67, 1992.

\bibitem{DF} Diehl, J.; Friz, Peter:\ Backward stochastic differential
equations with rough drivers. Annals of Probability (accepted, 2011).

\bibitem{MR2179357FS} Wendell~H. Fleming and H.~Mete Soner. 
\newblock {\em
Controlled {M}arkov processes and viscosity solutions}, volume~25 of \emph{%
Stochastic Modelling and Applied Probability}. \newblock Springer, New York,
second edition, 2006.

\bibitem{friz-victoir-book} Peter~K. Friz and Nicolas~B. Victoir. 
\newblock {\em Multidimensional stochastic processes as rough paths: theory and
  applications}. \newblock Cambridge Studies in Advanced Mathematics, 120.
Cambridge University Press, Cambridge, 2010. \newblock

\bibitem{MR1119185} Giga, Y.; Goto, S.; Ishii, H.; Sato, M.-H. \newblock %
Comparison principle and convexity preserving properties for singular
degenerate parabolic equations on unbounded domains. 
\newblock {\em Indiana
Univ. Math. J}, 40, no. 2, 443--470, 1991.

\bibitem{MR1959710} P.-L. Lions and P.~E. Souganidis. \newblock Viscosity
solutions of fully nonlinear stochastic partial differential equations. %
\newblock {\em S\=urikaisekikenky\=usho K\=oky\=uroku}, (1287):58--65, 2002. %
\newblock Viscosity solutions of differential equations and related topics
(Japanese) (Kyoto, 2001).

\bibitem{MR1647162} Pierre-Louis Lions and Panagiotis~E. Souganidis. %
\newblock Fully nonlinear stochastic partial differential equations. %
\newblock {\em C. R. Acad. Sci. Paris S\'er. I Math.}, 326(9):1085--1092,
1998.

\bibitem{MR1659958} Pierre-Louis Lions and Panagiotis~E. Souganidis. %
\newblock Fully nonlinear stochastic partial differential equations:
non-smooth equations and applications. 
\newblock {\em C. R. Acad. Sci. Paris
S\'er. I Math.}, 327(8):735--741, 1998.

\bibitem{MR1799099} Pierre-Louis Lions and Panagiotis~E. Souganidis. %
\newblock Fully nonlinear stochastic pde with semilinear stochastic
dependence. \newblock {\em C. R. Acad. Sci. Paris S\'er. I Math.},
331(8):617--624, 2000.

\bibitem{MR1807189} Pierre-Louis Lions and Panagiotis~E. Souganidis. %
\newblock Uniqueness of weak solutions of fully nonlinear stochastic partial
differential equations. \newblock {\em C. R. Acad. Sci. Paris S\'er. I Math.}%
, 331(10):783--790, 2000.

\bibitem{lyons-98} Terry Lyons. \newblock Differential equations driven by
rough signals. \newblock {\em Rev. Mat. Iberoamericana}, 14(2):215--310,
1998.

\bibitem{lyons-qian-02} Terry Lyons and Zhongmin Qian. 
\newblock {\em System
{C}ontrol and {R}ough {P}aths}. \newblock Oxford University Press, 2002. %
\newblock Oxford Mathematical Monographs.

\bibitem{MR2314753} Terry~J. Lyons, Michael Caruana, and Thierry L{\'{e}}vy. %
\newblock {\em Differential equations driven by rough paths}, volume 1908 of 
\emph{Lecture Notes in Mathematics}. \newblock Springer, Berlin, 2007. %
\newblock Lectures from the 34th Summer School on Probability Theory held in
Saint-Flour, July 6--24, 2004, With an introduction concerning the Summer
School by Jean Picard.

\bibitem{MR1258986} {\'{E}}tienne Pardoux and Shi~Ge Peng. \newblock %
Backward doubly stochastic differential equations and systems of quasilinear 
{SPDE}s. \newblock {\em Probab. Theory Related Fields}, 98(2):209--227, 1994.
\end{thebibliography}
\end{document}